\documentclass[12pt]{article}
\usepackage[margin=1in]{geometry}
\usepackage{booktabs}
\usepackage{xcolor}
\RequirePackage[sectionbib,numbers,sort&compress]{natbib}
\usepackage{amsmath, mathtools, amssymb, amsthm}
\newcommand{\mathbbm}[1]{\mathbf{#1}}
\newtheorem{theorem}{Theorem}
\newtheorem{lemma}{Lemma}
\newtheorem{proposition}{Proposition}

\theoremstyle{definition}

\usepackage{pdfpages}
\usepackage[hidelinks]{hyperref}
\usepackage{currfile}
\newtheorem{assumption}{Assumption}
\title{Global minimax rates for heterogeneous treatment effect estimation}

\author{
Shuhan Liu \quad Matey Neykov\textsuperscript{*}\\[0.5em]
\small Department of Statistics and Data Science, Northwestern University\\
\small {Evanston, Illinois, 60208, USA}\\[0.3em] \small \textsuperscript{*}Corresponding author: \texttt{mneykov@northwestern.edu} }
\date{}

\begin{document}

\maketitle
\begin{abstract}
Recent work established minimax rates for pointwise estimation of the conditional average treatment effect (CATE), but the fundamental limits under global losses remained unresolved. We derive matching minimax lower and upper bounds under the $L_2$ and $L_\infty$ risks on certain covariate-density and support classes. For the lower bounds, we develop a many-fuzzy-hypotheses method that replaces each fixed alternative in a multiple-hypothesis construction with a mixture over nuisance functions that leaves its CATE unchanged. A localized product-mixture $\chi^2$ bound is used to control the divergence of each resulting mixture from a common baseline, enabling multiple well-separated CATE hypotheses to be handled simultaneously. For the upper bounds, we analyze a modified local polynomial R-learner. The $L_2$ argument extends the pointwise bias--variance analysis, while the $L_\infty$ argument combines concentration inequalities for U-statistics with a continuity argument enabled by the use of a continuous kernel and additional regularity of the second-order basis. The resulting rates show that the pointwise elbow phenomenon persists globally: the $L_2$ transition threshold is unchanged, while the $L_\infty$ rate has an additional logarithmic factor and a transition threshold adjusted by $1-\log\log n/\log n$.
\end{abstract}


\section{Introduction}

Treatment effect heterogeneity arises across a wide range of scientific disciplines \cite{heckman1997making, rothwell1995can, green2012modeling}. Individuals may respond differently to the same intervention because of differences in their baseline characteristics, and understanding this heterogeneity is crucial for informed decision-making \cite{manski2004statistical}. In medicine, characterizing heterogeneous treatment effects is central to precision medicine, where treatments are tailored to individual patient profiles rather than selected solely on the basis of population-level averages \cite{kosorok2019precision, kent2018personalized}. Similar considerations arise in economics, education, and public policy, where knowledge of treatment effect heterogeneity can inform the design of targeted interventions and the allocation of limited resources \cite{kitagawa2018should, athey2021policy}. Consequently, developing reliable methods for estimating heterogeneous treatment effects has become a major focus of modern causal inference.

Under the potential-outcomes framework, treatment effect heterogeneity is commonly characterized by the conditional average treatment effect (CATE). Let $X$ denote a vector of covariates with density $f$ supported on $\mathcal{X}\subset[0,1]^d$, and let $Y(1)$ and $Y(0)$ denote the potential outcomes under treatment and control, respectively. The CATE is defined as
\[
 \mathbb{E} ( Y(1) - Y(0) \mid X=x).
\]
Let $A \in \{0,1\}$ denote the treatment indicator, and $Y$ the observed outcome. Define the propensity score $\pi(x)= \Pr(A=1\mid X=x)$, and the conditional outcome regression functions $\mu_a(x)=\mathbb{E}(Y\mid X=x,A=a)$ for $a\in \{0,1\}$. Denote the observed data by $Z=(X,A,Y)$. Suppose that $Z_1, \ldots, Z_n$ are independent and identically distributed observations from a distribution $P$. Under the standard consistency, ignorability, and positivity assumptions, namely,
\[
Y = AY(1) + (1-A)Y(0), \quad
(Y(0), Y(1)) \perp A \mid X,
\]
and $\varepsilon < \pi(X) < 1-\varepsilon$ almost surely for some $\varepsilon>0$, the CATE can be identified by
\[
\tau_P(x) = \mu_1(x)-\mu_0(x).
\]
Throughout the paper, we suppress the subscript $P$ whenever the underlying distribution is clear from context.

Estimating the CATE is fundamentally challenging because only one potential outcome is observed for each individual, and as a result, individual treatment effects are never observed directly. Early approaches to heterogeneous treatment effect estimation relied on relatively restrictive structural assumptions. Examples include partially linear and varying-coefficient models \cite{robinson1988root, robins1992estimating}, which encompass the special case of a constant treatment effect, as well as parametric and semiparametric structural nested models \cite{vansteelandt2014structural}. More recently, attention has shifted toward flexible estimators capable of accommodating complex treatment effect heterogeneity. Prominent examples include meta-learners \cite{kunzel2019metalearners}, causal forests \cite{wager2018estimation}, and orthogonalized estimators such as the R-learner and DR-learner \cite{foster2023orthogonal, nie2021quasi, kennedy2023towards, kennedy2024minimax}. We refer readers to \cite{kennedy2023towards, kennedy2024minimax} and Chapter 4 of \cite{wager2024causal} for comprehensive overviews of modern methods for CATE estimation.

Early theoretical analyses primarily focused on establishing convergence rates for specific estimators rather than characterizing minimax-optimal performance \cite{wager2018estimation, foster2023orthogonal, kennedy2023towards}. The fundamental statistical limits of CATE estimation remained poorly understood until Kennedy et al.\cite{kennedy2024minimax} developed a sharp minimax theory for pointwise CATE estimation. Specifically, they considered a nonparametric model in which $\pi$ is $\alpha$-smooth, $\mu_0$ is $\beta$-smooth, the CATE $\tau$ is $\gamma$-smooth, and the covariate density $f$ is bounded above. For estimating $\tau$ at a fixed point $x_0 \in [0,1]^d$ under the pointwise risk
\[
    \mathbb{E}|\hat{\tau}(x_0) - \tau(x_0)|,
\]
they derived the minimax rate and constructed a local polynomial R-learner that attains it. A notable feature of their result is an elbow phenomenon governed by the average nuisance smoothness $s = (\alpha+\beta)/2$. When $s\geq\frac{d/4}{1+d/2\gamma}$, the minimax rate coincides with the classical nonparametric regression rate, while below this threshold, the minimax rate deteriorates and interpolates between the rates for nonparametric regression and functional estimation. This behavior reflects the hybrid nature of the CATE as both a regression function and a functional of the observed data distribution.

While pointwise rates characterize the difficulty of estimating the CATE at a fixed covariate value, many applications require accurate estimation over the entire covariate space. In such settings, global risk criteria provide a more comprehensive assessment of estimator performance. The $L_2$ risk, which measures the integrated estimation error over the covariate space, is a natural criterion for evaluating overall estimation accuracy. Indeed, many existing CATE estimators are motivated by or analyzed through squared-error objectives \cite{nie2021quasi, kunzel2019metalearners, foster2023orthogonal}. In contrast, the $L_\infty$ risk controls worst-case performance over the covariate space and guards against large local errors. Such uniform guarantees are particularly relevant when treatment decisions are individualized. These considerations motivate the study of minimax rates for CATE estimation under both $L_2$ and $L_\infty$ risks. Formally, we consider
\[
    \mathbb{E}\|\hat{\tau}- \tau\|_{L_2(\mathcal{X})}^2, \quad
    \mathbb{E}\|\hat{\tau}- \tau\|_{L_\infty(\mathcal{X})}^2,
\]
where
\[
    \|g\|_{L_2(\mathcal{X})}^2 = \int_{\mathcal{X}} g(x)^2 dx ,
    \quad
    \|g\|_{L_\infty(\mathcal{X})} = \sup_{x \in \mathcal{X}} |g(x)| .
\]
Under this definition, the risks depend only on the support $\mathcal{X}$ and not on the covariate density $f$. When $f$ is bounded above and away from zero on $\mathcal{X}$, the corresponding risks are equivalent, up to multiplicative constants, to those defined with respect to the marginal distribution of $X$.

The primary goal of this paper is to establish the sharp minimax rates for CATE estimation under the global $L_2(\mathcal{X})$ and $L_{\infty}(\mathcal{X})$ risks, and to determine whether the elbow phenomenon of \cite{kennedy2024minimax} persists under these global risk criteria. The global lower bounds require more than a direct application of either the classical two-fuzzy-hypothesis method or a standard multiple-testing argument. The former compares only two mixture distributions and therefore cannot directly exploit a large collection of CATE functions that are well separated under a global loss, whereas the fixed alternatives used in the latter can be too easily distinguished in the low-nuisance-smoothness regime. We combine these two ideas by replacing each fixed alternative in a global packing with a fuzzy hypothesis that mixes over nuisance perturbations while keeping the CATE fixed. A new localized product-mixture $\chi^2$ bound is proposed that tracks only the regions where each alternative differs from the baseline and provides the divergence control needed to establish lower bounds under both the $L_2$ and $L_\infty$ risks. Section~\ref{sec::lower} formalizes this construction and applies it to establish the two global minimax lower bounds. Subsequently, Section~\ref{sec::upper} analyzes a modified local polynomial R-learner based on \cite{kennedy2024minimax}. By incorporating a continuous kernel and imposing additional regularity conditions on the second-order basis, we obtain uniform control over the covariate space through concentration inequalities for U-statistics and a continuity argument, thereby establishing matching upper bounds. Our results confirm that the elbow phenomenon persists under both global risk criteria. Under the $L_2(\mathcal{X})$ risk, the transition point coincides with that of pointwise estimation. Under the $L_\infty(\mathcal{X})$ risk, the stronger uniform loss criterion induces an additional logarithmic factor in the minimax rate, accompanied by a slight shift in the transition boundary. Together, the lower and upper bounds establish the sharp minimax rates under both global risk criteria on the covariate-density and support classes described in Appendix~\ref{subsec:density_and_construction}, with the stated upper-bound conditions.

Throughout the paper, $\mathbf{e}_i$ denotes the $i$-th canonical basis vector in coordinate vector space. For a vector $a$, let $\|a\|_0$, $\|a\|_1$, $\|a\|_2$, $\|a\|_\infty$ denote the $\ell_0$, $\ell_1$, Euclidean, and sup norms, respectively, and let $[a]_j$ denote the $j$-th element of $a$. For a matrix $A$, let $\|A\|$ denote the operator norm and $\|A\|_{\infty}$ the maximum row-sum norm. For a symmetric real matrix $A$, let $A^k$ denote its $k$-th matrix power, $[A]_{i,j}$ denote the $(i,j)$-th entry of $A$, and ${\lambda}_{\min}(A)$ and ${\lambda}_{\max}(A)$ denote its smallest and largest eigenvalues, respectively. For any real square matrix $A$, let $\sigma_{\min}(A)=\sqrt{\lambda_{\min}(A^TA)}$ denote its smallest singular value. For a measurable function $g$ and measure $Q$, let $\|g\|_Q^2= \int g(z)^2 dQ(z)$  denote the squared $L_2(Q)$ norm. When $Q$ is the uniform measure on the underlying domain, the subscript is omitted. We also write $\|g\|_{\mathcal{L}_\infty(\mathcal{Z})} = \sup_{z \in \mathcal{Z}} |g(z)|$. All smoothness assumptions are defined in terms of standard Hölder classes. We use $P$ to denote a probability distribution, $P^n$ its $n$-fold product measure, $\mathcal{P}$ a model class, $\Pr$ the probability, and $\mathbb{E}$ the expectation. Let $\mathbbm{1}\{B\}$ denote the indicator function of an event $B$, and write $\mathbbm{1}_{\mathcal{Z}}(z)$ as shorthand for $\mathbbm{1}\{z \in \mathcal{Z}\}$. For two sets $C$ and $D$, let $C \setminus D=C\cap D^c$ denote the set difference. For two sequences $a_n$ and $b_n$, write $a_n\asymp b_n$ if there exist constants $0<c<C<\infty$ such that $ c\cdot b_n \leq a_n \leq C \cdot b_n$ for all sufficiently large $n$. Write $a_n\lesssim b_n$ if $a_n \leq C b_n$ for some constant $C>0$, and $a_n\gtrsim b_n$ if $a_n \geq c b_n$ for some constant $c>0$. Finally, we use $\mathbb{P}$  and $\mathbb{U}$ for empirical averages and second-order U-statistics, that is $\mathbb{P}(g) = \frac{1}{n}\sum_{i=1}^n g(Z_i)$, and $\mathbb{U}(g) = \frac{1}{n(n-1)}\sum_{i\neq j} g(Z_i,Z_j)$.

\section{Lower Bound} \label{sec::lower}
Following \cite{kennedy2024minimax}, we consider binary outcomes and use the parametrization $(f, \pi, \mu_0, \tau)$. Throughout this section, we focus on the regime $\alpha \geq \beta$. Constructions for the complementary regime and for alternative parametrizations follow by adapting the arguments in Appendix~A of \cite{kennedy2024minimax} to the $L_2(\mathcal{X})$ and $L_\infty(\mathcal{X})$ risks.

Let $\mathcal{P}$ denote the class of distributions satisfying the following conditions:
\begin{itemize}
\item[1.] The covariate density $f$ is uniformly bounded above on the support $\mathcal{X} \subset [0,1]^d$.
\item[2.] $\pi$ is $\alpha$-smooth, $\mu_0$ is $\beta$-smooth, and $\tau$ is $\gamma$-smooth.
\end{itemize}
Define $s = (\alpha + \beta)/2$. The following theorem establishes minimax lower bounds under the $L_2(\mathcal{X})$ and $L_\infty(\mathcal{X})$ risks.
\begin{theorem} \label{thm::lower}
For sufficiently large $n$,
\[
\inf_{\hat{\tau}} \sup_{P \in \mathcal{P}} \mathbb{E}_{P^n}\|\hat{\tau} - \tau\|_{L_2(\mathcal{X})}^2
\gtrsim
\begin{cases}
n^{-\frac{2}{d/2\gamma+d/4s+1}} & s < \frac{d/4}{1+d/2\gamma}\\
n^{-\frac{2}{2+d/\gamma}}& s \geq \frac{d/4}{1+d/2\gamma}\\
\end{cases}.
\]
Moreover,
\[
\inf_{\hat{\tau}} \sup_{P \in \mathcal{P}} \mathbb{E}_{P^n}\|\hat{\tau} - \tau\|_{L_\infty(\mathcal{X})}^2
\gtrsim
\begin{cases}
{n}^{-\frac{2}{d/2\gamma+d/4s+1}}
({\log n})^{\frac{1}{d/2\gamma+d/4s+1}}
&
s < \frac{d/4}{1+d/2\gamma}(1-\frac{\log \log n}{\log n})\\
(\frac{\log n}{n})^{\frac{2}{2+d/\gamma}} &
s \geq \frac{d/4}{1+d/2\gamma}(1-\frac{\log \log n}{\log n })\\
\end{cases}.
\]
\end{theorem}
The lower bound under the $L_2(\mathcal{X})$ risk exhibits the same elbow phenomenon identified by \cite{kennedy2024minimax} for pointwise CATE estimation. In particular, when $s < \frac{d/4}{1+d/2\gamma}$, the classical nonparametric regression rate is unattainable, and the minimax risk depends jointly on the smoothness of the target function $\gamma$ and the nuisance smoothness $s$. Once $s$ exceeds the threshold $\frac{d/4}{1+d/2\gamma}$, the lower bound coincides with the classical nonparametric regression rate, and the nuisance smoothness no longer determines the fundamental difficulty of the problem.

The lower bound under the $L_\infty(\mathcal{X})$ risk exhibits the same polynomial rate as the $L_2(\mathcal{X})$ risk, but incurs an additional logarithmic factor. This reflects the increased difficulty of uniform estimation over the covariate space, where the estimation error must be controlled simultaneously across the entire domain $\mathcal{X}$. Correspondingly, the transition between the low- and high-smoothness regimes is shifted by a logarithmic correction of order
\[
\frac{\log \log n}{\log n}.
\]

For both risks, the lower bounds in the high-smoothness regime coincide with the classical minimax rates for nonparametric regression. These rates can be obtained via standard lower-bound constructions for nonparametric regression, such as those in Section 2.6 of \cite{Tsybakov2009intro}. We therefore focus on the low-smoothness regime in the remainder of this section, where the classical nonparametric regression lower bound is no longer optimal and a many-fuzzy-hypothesis construction is applied to establish the minimax rate.

\subsection{Lower bounds based on many fuzzy hypotheses}
A key ingredient in the lower-bound construction of \cite{kennedy2024minimax} is the method of two fuzzy hypotheses. Under this framework, one constructs two mixture distributions over a prior $\pi$,
\[
\int P_{\lambda}^n\,d\pi(\lambda)
\qquad\text{and}\qquad
\int Q_{\lambda}^n\,d\pi(\lambda),
\]
and controls the difficulty of distinguishing between them through bounds on the Hellinger distance, following \cite{robins2009semiparametric}.

Global losses generally require a collection of well-separated target values rather than a single pair. One way to incorporate fuzzy hypotheses into such a construction is through Assouad's lemma. In particular, \cite{bonvini2023minimax} construct an Assouad hypercube in which each vertex is mixed over nuisance perturbations while its CATE remains fixed, and use this construction to derive a lower bound for CATE level-set estimation. They further note that the same construction can be applied to derive a lower bound for CATE estimation under the integrated $L_2(\mathcal{X})$ risk. Thus, an Assouad-style fuzzy-hypothesis argument provides a valid route to the $L_2(\mathcal X)$ lower bound through pairwise Hellinger control between neighboring vertices.

However, Assouad's lemma is limited to losses with a suitable coordinatewise structure and does not extend naturally to the $L_\infty(\mathcal X)$ risk \cite{Tsybakov2009intro}. To treat the $L_2(\mathcal X)$ and $L_\infty(\mathcal X)$ risks within a common framework, we instead combine fuzzy hypotheses with a common-baseline multiple-testing argument. In our construction, each fixed alternative in a standard multiple-hypothesis construction is replaced by a mixture over nuisance perturbations that leaves its CATE unchanged. Because the testing argument compares each resulting mixture with a common baseline mixture, it is naturally formulated through likelihood ratios and calls for control of divergences such as the Kullback--Leibler or $\chi^2$ divergence.

Theorem~\ref{lemma::chi} then provides the required localized $\chi^2$-divergence bound for mixtures of product measures, building on the partition-based approach of Theorem~2.1 of \cite{robins2009semiparametric}. A closely related bound appears in Proposition~2.6 of \cite{liu2021adaptive} in the context of adaptive functional estimation, although its proof appears to require further justification. We give an alternative proof by adapting the conditional mixture argument of \cite{robins2009semiparametric} and using the negative association of multinomial cell counts. The resulting bound depends only on the partition cells where the two mixtures differ, allowing multiple fuzzy alternatives with different localization patterns to be compared with a common baseline.

The main distinction with the existing literature lies in how this localized divergence control is used. Whereas \cite{liu2021adaptive} combines it with a constrained risk inequality for a two-hypothesis lower bound, Theorem~\ref{lemma::lower-bound} below combines it with a multiple-hypothesis testing argument based on a collection of well-separated fuzzy hypotheses. Together, these results provide a unified approach to the lower bounds under both the $L_2(\mathcal X)$ and $L_\infty(\mathcal X)$ risks. To the best of our knowledge, this common-baseline multiple-testing formulation of many fuzzy hypotheses, and its application to the $L_\infty$ risk, have not previously appeared in the minimax lower-bound literature. Although developed here for CATE estimation, the underlying strategy may be useful for a broader class of nonparametric and semiparametric estimation problems.

 \begin{theorem} \label{lemma::chi}
Let $P_\lambda$ and $Q_\lambda$ be probability distributions from $\mathcal{P}$ indexed by
     \[
     \lambda = (\lambda_1, \cdots, \lambda_k) \in
     \Lambda = \Lambda_1 \times \cdots \times \Lambda_k,
     \]
and let $P_\lambda^n$ and $Q_\lambda^n$ denote the corresponding $n$-fold product measures. Let $\mathcal{Z}=\cup_{j=1}^k \mathcal{Z}_j$ denote a partition of the sample space. Assume:
     \begin{itemize}
         \item[1.] For every $\lambda \in \Lambda$, $P_\lambda(\mathcal{Z}_j) = Q_\lambda(\mathcal{Z}_j) = p_j$.
         \item[2.] The conditional distributions obtained by restricting $P_\lambda$ and $Q_\lambda$ to $\mathcal{Z}_j$, denoted by $ P_{j, \lambda_j}$ and $Q_{j, \lambda_j}$, and defined through the densities
\[dP_{j, \lambda_j}= \frac{\mathbbm{1}_{\mathcal{Z}_j} d P_\lambda}{p_j}, \quad dQ_{j,\lambda_j}=\frac{\mathbbm{1}_{\mathcal{Z}_j} d Q_\lambda}{p_j},
         \]
depend only on $\lambda_j$ but not on $\lambda_\ell$ for all $\ell \neq j$. Furthermore, there exists a subset $S \subseteq \{1,\cdots,k\}$ such that for all $j \notin S$, $P_{j, \lambda_j} = Q_{j, \lambda_j}.$ In other words, $ P_{j, \lambda_j}$ and $Q_{j, \lambda_j}$ may differ only for $j \in S$.
     \end{itemize}
Let $p_\lambda$ and $q_\lambda$ denote the densities of $P_\lambda$ and $Q_ \lambda$ with respect to a dominating measure $\nu$. For a product probability measure $\pi = \pi_1 \otimes \cdots \otimes \pi_k$, define $\bar{p} = \int p_\lambda d\pi(\lambda)$, $\bar{q} = \int q_\lambda d\pi(\lambda)$, and
     \begin{align*}
         \delta_1 &= \max_{j \in S} \sup_{\lambda \in \Lambda} \int_{\mathcal{Z}_j} \frac{(p_\lambda-\bar{p})^2}{p_\lambda p_j}d\nu\\
         \delta_2 &= \max_{j \in S} \sup_{\lambda \in \Lambda} \int_{\mathcal{Z}_j} \frac{(q_\lambda-p_\lambda)^2}{p_\lambda p_j}d\nu\\
         \delta_3 &= \max_{j \in S} \sup_{\lambda \in \Lambda} \int_{\mathcal{Z}_j} \frac{(\bar{q}-\bar{p})^2}{p_\lambda p_j}d\nu.
     \end{align*}
Suppose that $\bar{p}/p_\lambda \leq \varepsilon_1 < \infty$ for all $\lambda \in \Lambda$  and $np_j \max(1, \delta_1, \delta_2)\leq \varepsilon_2 < \infty $ for every $j \in S$. Then
     \[
    \chi^2
    \left(
    \int Q_\lambda^n d \pi(\lambda),
    \int P_\lambda^n d \pi(\lambda)
    \right)
    \leq
    \exp \left(
    C n \sum_{j \in S} p_j \left\{n \max_{j \in S} p_j (\delta_2^2+\delta_1 \delta_2) + \delta_3\right\}
    \right)
    -1,
    \]
where $C$ depends only on $\varepsilon_1$ and $\varepsilon_2$.
 \end{theorem}
\begin{proof}
See Appendix~\ref{app::chi}.
\end{proof}

The following theorem extends the classical multiple-hypothesis lower-bound argument, such as Theorem~2.6 of \cite{Tsybakov2009intro}, to fuzzy hypotheses. In the classical setting, each hypothesis is represented by a single distribution, and the targets are well separated while the divergence of every alternative from a common baseline is controlled. In the fuzzy-hypothesis setting, each fixed distribution is replaced by a mixture over a family of distributions indexed by $\lambda$, while the target $\mathbb S((P^{(i)})_\lambda)$ is required to remain invariant in $\lambda$. Thus, mixing preserves the separation among the targets, and controlling the divergence of each resulting mixture from the baseline mixture yields a minimax lower bound.

\begin{theorem} \label{lemma::lower-bound}
Let $\{(P^{(i)})_\lambda: \lambda \in \Lambda\}$, $i = 0, \cdots, M$, be families of distributions in $\mathcal P$, and let $\pi$ be the prior distribution introduced in Theorem~\ref{lemma::chi}. Let $\mathbb{S}$ denote the mapping from a probability model to the target, and $\operatorname{dist}(\cdot, \cdot)$ be a semi-metric on the target space. Suppose the following conditions hold:
\begin{itemize}
\item[1.]
For each $i=0, \cdots, M$,
\[
\mathbb{S}\bigl((P^{{(i)}})_\lambda\bigr)
=
\mathbb{S}\bigl((P^{{(i)}})_{\lambda'}\bigr),
\qquad
\forall\,\lambda,\lambda^\prime \in \Lambda,
\]
and denote this common value by $\mathbb{S}_{{(i)}}$.

\item[2.] The targets are well-separated in the sense that
\[
\operatorname{dist}(\mathbb{S}_{{(i)}},\mathbb{S}_{{(j)}}) \ge 2r > 0, \quad \forall\,0\le i<j\le M,
\]
for some $r >0$.

\item[3.]
For every $i = 1, \cdots,  M$, $ \int (P^{{(i)}})_\lambda^{n}\,d\pi(\lambda) \ll \int (P^{{(0)}})_\lambda^{n}\,d\pi(\lambda), $ and,
\[
\frac{1}{M}
\sum_{i=1}^M
\chi^2\!\left(
\int (P^{{(i)}})_\lambda^{ n} d\pi(\lambda),
\int (P^{{(0)}})_\lambda^{ n} d\pi(\lambda)
\right)
\le c M,
\]
for some $0<c\le \tfrac12$.
\end{itemize}

Then,
\[
\inf_{\widehat{\mathbb{S}}}
\sup_{P\in\mathcal P}
\mathbb E_{P^{ n}}
\!\left(
\operatorname{dist}\bigl(\widehat{\mathbb{S}},\mathbb{S}(P)\bigr)
\right)
\;\ge\;
\frac{r}{2}\Bigl(1-c-\frac1M\Bigr).
\]
\end{theorem}

\begin{proof}
See Appendix~\ref{app::lower-bound}.
\end{proof}

We now construct families of mixtures that remain sufficiently close in $\chi^2$-divergence while their targets are well separated under the $L_2(\mathcal{X})$ and $L_\infty(\mathcal{X})$ losses.

\subsection{Minimax lower bound under the $L_2(\mathcal{X})$ and $L_\infty(\mathcal{X})$ risks}
The construction builds on \cite{kennedy2024minimax}, but differs in one essential respect. Their fuzzy perturbations in $P_\lambda$ and $Q_\lambda$ are localized around a fixed point $x_0$, whereas ours are distributed across the covariate space. We then form the family of alternatives by combining $P_\lambda$ and $Q_\lambda$ over different regions of the sample space.

Formally, let \(\mathcal{C}_h(x)\) denote the cube of side length $h$ centered at \(x \in (0,1)^d\). Suppose that both $A$ and $Y$ are binary. Then each observation $Z$ takes values in $[0,1]^d \times \{0,1\} \times \{0,1\}$. We construct the following partition of the sample space.
\begin{itemize}
    \item Partition $[0,1]^d$ into $m^d$ cubes of equal side length $2h$ with midpoints  $s_1, \cdots,s_{m^d}$, so that $2^dh^d m^d= 1$.
    \item Within each cube $\mathcal{C}_{2h}(s_i)$, partition the inner cube $\mathcal{C}_{h}(s_i)$ into $k$ cubes of equal size with centers $s_{i1}, \cdots, s_{ik}$.
Denote these subcubes by $\mathcal{C}_{h/k^{1/d}}(s_{ij})$.
    \item Let $\mathcal{Z}_{i} = \mathcal{C}_{2h}(s_i) \times \{0,1\} \times \{0,1\}$ for $i=1, \cdots, m^d$, and for each $i$, define $\mathcal{Z}_{ij} = \mathcal{C}_{h/2k^{1/d}}(s_{ij}) \times \{0,1\} \times \{0,1\}$ for $j=1, \cdots, k$.
\end{itemize}

Let $B: \mathbb{R}^d \rightarrow \mathbb{R}$ be a $C^\infty$ function satisfying $B(x)=1$ for $x \in [-1/4, 1/4]^d$, and $B(x)=0$ for $x \notin [-1/2, 1/2]^d$. Let $\lambda =(\lambda_{ij})_{1 \leq i \leq m^d, 1 \leq j \leq k}$ denote a vector indexed by the pairs $(i,j)$ with entries $\lambda_{ij}$. Define
\begin{align*}
    \tau_h(x) =& h^\gamma \sum_{i=1}^{m^d}B(\frac{x-s_{i}}{2h}),\\
    \mu_{0\lambda}(x) =& \frac{1}{2} + (h/k^{1/d})^{\beta} \sum_{i=1}^{m^d} \sum_{j=1}^k \lambda_{ij} B(\frac{x-s_{ij}}{h/k^{1/d}}), \\
    \pi_\lambda(x) =& \frac{1}{2} + (h/k^{1/d})^\alpha \sum_{i=1}^{m^d} \sum_{j=1}^k \lambda_{ij} B(\frac{x-s_{ij}}{h/k^{1/d}}), \\
    f(x) =& 4^d \cdot \mathbbm{1}\left\{ x \in \cup_{i=1}^{m^d} \cup_{j=1}^k \mathcal{C}_{h/2k^{1/d}}(s_{ij}) \right\}.
\end{align*}
Under the model parametrization $(f, \pi, \mu_{0\lambda}, \tau)$, define the distributions $P_\lambda$ and $Q_\lambda$ by
\[
    p_\lambda = (f, 1/2, \mu_{0\lambda}-\tau_h/2, \tau_h), \quad
    q_\lambda = (f, \pi_\lambda, \mu_{0\lambda}, 0).
\]
We next construct a family of mixture distributions by combining $P_\lambda$ and $Q_\lambda$ across the sets $\mathcal{Z}_1, \cdots \mathcal{Z}_{m^d}$. For a mixing vector
\[
w = (w_1, \cdots, w_{m^d}) \in \{0,1\}^{m^d},
\]
define $(P^{w})_\lambda $ through the density
    \[
        (p^{w})_\lambda = \sum_{i=1}^{m^d} \left[(1-w_i) \mathbbm{1}_{\mathcal{Z}_i} p_\lambda +
        w_i \mathbbm{1}_{\mathcal{Z}_i}q_\lambda
        \right],
    \]
and denote the corresponding $n$-fold product measure by $(P^{w})_\lambda^n$. The vector $w^{(0)} = (0, \cdots, 0)^T$ corresponds to the baseline model $P_\lambda$, that is, $(P^{w^{(0)}})_\lambda = P_\lambda$.

Several observations are useful. First, whenever $P_\lambda$ and $Q_\lambda$ satisfy the conditions of Theorem~\ref{lemma::chi}, the same is true for $(P^w)_\lambda$. In particular, for every $\lambda \in \Lambda$,
\[
(P^w)_\lambda(\mathcal{Z}_{ij}) = P_\lambda(\mathcal{Z}_{ij}) = Q_\lambda(\mathcal{Z}_{ij}),
\]
and the conditional distribution on $\mathcal{Z}_{ij}$ depends only on the local coordinate $\lambda_{ij}$. Second, $(P^w)_\lambda$ differs from the baseline $P_\lambda$ only on those regions $\mathcal{Z}_{i}$ for which $w_i=1$, while coinciding with $P_\lambda$ on all regions with $w_{i}=0$. The latter regions play the same role as $[0,1]^d \setminus \mathcal{C}_{2h}(x_0)$ in \cite{kennedy2024minimax}, where the two distributions agree exactly. As a result, the index set $S$ in Theorem~\ref{lemma::chi} contains only those indices $i$ for which $w_i=1$, so that the resulting $\chi^2$-divergence from the baseline is controlled by the Hamming weight of $w$. Third, along the boundary of each $\mathcal{Z}_i$, $\tau_h(x)=0$, $\mu_{0 \lambda}(x) = \pi_\lambda(x) = 1/2$, and therefore $p_\lambda = q_\lambda$. This guarantees that $(P^w)_\lambda$ inherits the required smoothness and avoids introducing discontinuities at the interfaces between regions governed by $P_\lambda$ and $Q_\lambda$. Finally, since $\tau_h$ is independent of $\lambda$, the invariance condition in Theorem~\ref{lemma::lower-bound} holds automatically.

To apply Theorem~\ref{lemma::lower-bound}, we next construct suitable mixing vectors $w^{(0)}, \ldots, w^{(M)}$ and choose $h$ and $k$ so that the targets $\tau(P^{w^{(i)}})$ are well separated, while each mixture $\int(P^{w^{(i)}})_\lambda^n d\pi(\lambda)$ remains sufficiently close in $\chi^2$-divergence to the baseline $\int(P^{w^{(0)}})_\lambda^n d\pi(\lambda)$. The following propositions establish these properties for the two losses.
\begin{proposition} \label{prop::L2}
Let $\pi$ be the uniform distribution on $\{-1,1\}^{m^d \times k}$, and assume $m^d>8$. Let ${w}^{(0)}, {w}^{(1)}, \cdots, {w}^{(M)} \in \{0,1\}^{m^d}$ with ${w}^{(0)} = (0, \cdots, 0)$, and suppose that
    \[
    d_H({w}^{(\ell)}, {w}^{(\ell^{\prime})}) \geq \frac{m^d}{8}, \quad \forall \ 0 \leq \ell < \ell^\prime \leq M,
    \]
where $d_H(\cdot, \cdot)$ denotes the Hamming distance. Let $M \geq 2^{m^d/8}$ as guaranteed by the Varshamov-Gilbert bound (Lemma 2.9 from \cite{Tsybakov2009intro}).
    \begin{itemize}
        \item[(1.)] Suppose $ h^\gamma = 4(h/k^{1/d})^{2s}$,  $h \leq 1/4$, $h^\gamma + 2(h/k^{1/d})^\beta \leq 1-4\epsilon_1$ for some $\varepsilon_1 \in (0,1/4)$, and
$n(4h)^d\max(1, 2\|B\|^2/\varepsilon_1)/k  \leq \varepsilon_2$. Then $\chi^2(  \int (P^{w^{(\ell)}})_\lambda^n d \pi(\lambda), \int (P^{w^{(0)}})_\lambda^n d \pi(\lambda) ) $ is bounded by
        \begin{equation} \label{eq::chi}
            \exp \left(
            4C  \left( \frac{4^d\|B\|^2}{\varepsilon_1}\right)^2 \left(\ \frac{n^2\|w^{(\ell)}\|_0  h^{2d}}{k}\right) \{ (h/k^{1/d})^{4s} + (h/k^{1/d})^{4\alpha} \}
            \right) -1
        \end{equation}
where  $C$ depends only on $\varepsilon_1$ and $\varepsilon_2$.
        \item[(2.)] Let $C^{*} = C \cdot 2^{4d-d/s-1}  \left( {\|B\|^2}/{\varepsilon_1}\right)^2$, choose $a <   \min\{ 1, {\log 2}/{16C^{*}}\}$, and set
$h = ({a}/{n^2})^{{1}/({d+d\gamma/2s+2\gamma})}$. Then, for sufficiently large $n$,
        \[
        \inf_{\hat{\tau}} \sup_{P \in \mathcal{P}} \mathbb{E}_{P^n}\|\hat{\tau} - \tau\|_{L_2(\mathcal{X})}^2
        \gtrsim  n^{-\frac{2}{d/2\gamma+d/4s+1}}.
        \]
    \end{itemize}
\end{proposition}

\begin{proof}
For every region $\mathcal{Z}_i$ satisfying $w^{(\ell)}_{i}=0$,  the mixture distribution $(P^{w^{(\ell)}})_\lambda$ coincides with the baseline distribution $(P^{w^{(0)}})_\lambda$. Consequently, only the subregions $\mathcal{Z}_{ij}$ contained in regions $\mathcal{Z}_{i}$ for which $w^{(\ell)}_{i}=1$ contribute to the $\chi^2$-divergence. On each such subregion, the pair $\left( (P^{w^{(\ell)}})_\lambda, (P^{w^{(0)}})_\lambda \right)$ coincides with the pair $(Q_\lambda, P_\lambda)$ studied in \cite{kennedy2024minimax} up to a normalizing constant depending only on $d$. Therefore, following the same calculations, the quantities $p_{ij}$, $\delta_1$, $\delta_2$, and $\delta_3$ appearing in Theorem~\ref{lemma::chi}  satisfy
    \[
    p_{ij} =  \frac{2^dh^d}{k}, \quad
    \delta_1 \leq \left( \frac{2^d \|B\|^2}{\varepsilon_1}\right) (h/k^{1/d})^{2\beta}, \quad
    \delta_2 \leq \left( \frac{2^{d+1} \|B\|^2}{\varepsilon_1}\right) (h/k^{1/d})^{2\alpha}, \quad
    \delta_3 =0.
    \]
For the partition of the sample space $\cup_{i=1}^{m^d} \cup_{j=1}^k \mathcal{Z}_{ij}$, Theorem~\ref{lemma::chi} immediately yields the bound in (\ref{eq::chi}).

Next, using $\|w^{(\ell)}\|_0 \leq m^d$ together with  $h^\gamma = 4(h/k^{1/d})^{2s}$, we obtain
     \begin{align*}
        \frac{1}{M}
        \sum_{\ell=1}^M
        \chi^2\!\left(
        \int (P^{w^{(\ell)}})_\lambda^{ n}
        \right. & \left.d\pi(\lambda),
        \int (P^{w^{(0)}})_\lambda^{ n} d\pi(\lambda)
        \right) \\
        \leq &
        \exp  \left(
        4C  \left( \frac{4^d\|B\|^2}{\varepsilon_1}\right)^2 \left(\ \frac{n^2 m^d h^{2d}}{k}\right) \{ (h/k^{1/d})^{4s} + (h/k^{1/d})^{4\alpha} \}
        \right) -1 \\
        \leq&
        \exp(C^{*} m^d n^2 h^{d+d\gamma/2s+2\gamma})-1\\
        = &
        \exp\left(
        a C^{*}m^d
        \right) - 1
        < 2^{\frac{m^d}{16}}
        < \frac{1}{4} 2^{m^d/8}
        \leq \frac{1}{4}M
     \end{align*}
for sufficiently large $n$. Let $\tau_w$ denote the CATE corresponding to $(P^w)_\lambda$. By construction,
    \[
    \tau_w = \sum_{i=1}^{m^d}(1-w_i) \tau_h \mathbbm{1}_{\mathcal{C}_{2h}(s_i)}.
    \]
For any $0 \leq \ell < \ell^\prime \leq M$, since $d_H(w^{(\ell)}, w^{(\ell^\prime)}) \geq m^d/8$
    \begin{align*}
    \| \tau_{w^{(\ell)}} - \tau_{w^{(\ell^\prime)}} \|_{L_2(\mathcal{X})}
    = (
    \sum_{i: w_{i}^{(\ell)} \neq w_{i}^{(\ell^\prime)}} \sum_{j=1}^k
    \int_{\mathcal{X}_{ij}} h^{2\gamma} dx
    )^{1/2}\geq \left(
    \frac{m^d}{8} \cdot k \cdot h^{2\gamma} \cdot \frac{h^d}{2^dk}
    \right)^{1/2} =: 2r
    \asymp h^\gamma,
    \end{align*}
where $\mathcal{X}_{ij} = \mathcal{C}_{h/(2k^{1/d})}(s_{ij})$. Applying Theorem~\ref{lemma::lower-bound} to the $L_2(\mathcal X)$ norm and Jensen's inequality gives
    \begin{align*}
        \inf_{\hat{\tau}} \sup_{P \in \mathcal{P}} \mathbb{E}_{P^n}\|\hat{\tau} - \tau\|_{L_2(\mathcal{X})}^2 \geq \left(\inf_{\hat{\tau}} \sup_{P \in \mathcal P}
        \mathbb E_{P^n}\|\hat\tau-\tau\|_{L_2(\mathcal X)}\right)^2
        \gtrsim {r^2}
        \asymp h^{2\gamma}
        \asymp n^{-\frac{2}{d/2\gamma+d/4s+1}},
     \end{align*}
which completes the proof.
\end{proof}

The following proposition gives the lower-bound construction for the $L_\infty(\mathcal{X})$ risk.
\begin{proposition} \label{prop::Linf}
Let $\pi$ be the uniform distribution over $\{-1,1\}^{m^d \times k}$, and set $M=m^d$ and define $w^{(i)} = \mathbf{e}_i \in \mathbb{R}^{m^d}$ for $i=1, \ldots, M$.
    \begin{itemize}
        \item[(1.)] Suppose $ h^\gamma = 4(h/k^{1/d})^{2s}$,  $h \leq 1/4$, $h^\gamma + 2(h/k^{1/d})^\beta \leq 1-4\epsilon_1$ for some $\varepsilon_1 \in (0,1/4)$, and $n(4h)^d\max(1, 2\|B\|^2/\varepsilon_1)/k  \leq \varepsilon_2$.
Then $\chi^2(  \int (P^{w^{(\ell)}})_\lambda^n d \pi(\lambda), \int (P^{w^{(0)}})_\lambda^n d \pi(\lambda) )$ is bounded by
        \begin{equation} \label{eq::chi2}
            \exp \left(
            4C  \left( \frac{4^d\|B\|^2}{\varepsilon_1}\right)^2 \left(\ \frac{n^2  h^{2d}}{k}\right) \{ (h/k^{1/d})^{4s} + (h/k^{1/d})^{4\alpha} \}
            \right) -1
        \end{equation}
where  $C$ depends only on $\varepsilon_1$ and $\varepsilon_2$.
        \item[(2.)] Let $C^{*} = C \cdot 2^{4d-d/s-1}  \left( {\|B\|^2}/{\varepsilon_1}\right)^2$, choose
$0<b <\min\{1, {d}/{(C^{*}d/2+C^{*}d\gamma/4s+C^{*}\gamma)}\}$, and set $h = n^{- 1/({d/2+d\gamma/4s+\gamma})}(b \log n)^{{1}/({d+d\gamma/2s+2\gamma})}$. For sufficiently large $n$,
        \[
        \inf_{\hat{\tau}} \sup_{P \in \mathcal{P}} \mathbb{E}_{P^n}\|\hat{\tau} - \tau\|_{L_\infty(\mathcal{X})}^2
        \gtrsim  n^{-\frac{2}{d/2\gamma+d/4s+1}} (\log n)^{\frac{1}{d/2\gamma+d/4s+1}}.
        \]
    \end{itemize}
\end{proposition}
\begin{proof}
The proof of part (1) follows the same argument as Proposition~\ref{prop::L2}. The only difference is that the present construction satisfies $\|w^{(\ell)}\|_0=1$ for $\ell = 1, \cdots, m^d$. With the selected $ w^{(i)}$,
    \[
        2 r :=\|\tau_{w^{(\ell)}} -  \tau_{w^{(\ell^\prime)}}\|_{L_\infty(\mathcal{X})} = h^\gamma, \quad
         \forall \ 0 \leq \ell < \ell^{\prime} \leq m^d.
    \]
Furthermore, part (1) implies
    \begin{align*}
        \frac{1}{M}
        \sum_{\ell=1}^M
        \chi^2\!\left(
        \int (P^{w^{(\ell)}})_\lambda^{ n} d\pi(\lambda),
        \right. & \left.
        \int (P^{w^{(0)}})_\lambda^{ n} d\pi(\lambda)
        \right) \\
        \leq &
        \exp  \left(
        4C  \left( \frac{4^d\|B\|^2}{\varepsilon_1}\right)^2 \left(\ \frac{n^2 h^{2d} }{k}\right) \{ (h/k^{1/d})^{4s} + (h/k^{1/d})^{4\alpha} \}
        \right) -1 \\
        \leq &
        \exp(C^{*}n^2 {h^{d+{d\gamma}/{2s}+2\gamma}} ) - 1  \\
        = &  n^{bC^{*}} -1 \leq \frac{1}{4}  2^{-d} h^{-d} = \frac{1}{4} m^d
     \end{align*}
for sufficiently large $n$. Applying Theorem~\ref{lemma::lower-bound} to the $L_\infty(\mathcal X)$ norm and then Jensen's inequality as above yields
     \[
        \inf_{\hat{\tau}} \sup_{P \in \mathcal{P}} \mathbb{E}_{P^n}\|\hat{\tau} - \tau\|_{L_\infty(\mathcal{X})}^2
        \gtrsim {r^2}  \asymp h^{2\gamma} \asymp n^{-\frac{2}{d/2\gamma+d/4s+1}} (\log n)^{\frac{1}{d/2\gamma+d/4s+1}}.
    \]
\end{proof}
In particular, the lower bounds in Propositions~\ref{prop::L2} and \ref{prop::Linf} apply in the parameter regimes determined by the condition $nh^d/k\lesssim 1$. Substituting the corresponding choices of $h$ and $k$ into this condition yields
\[
s<\frac{d/4}{1+d/(2\gamma)}
\]
for the $L_2(\mathcal X)$ risk, and
\[
s<
\frac{d/4}{1+d/(2\gamma)}
\left(1-\frac{\log\log n}{\log n}\right)
\]
for the $L_\infty(\mathcal X)$ risk. The logarithmic adjustment in the second condition results from the additional $\log n$ factor in the choice of $h$ for the $L_\infty(\mathcal X)$ lower bound.

On the other hand, one may construct a model class that perturbs only $\tau$, while fixing $\mu_0=\pi=1/2$. Classical minimax lower bounds for nonparametric regression (see, for example, Section~2.6 of \cite{Tsybakov2009intro}) then imply
\[
\inf_{\hat{\tau}}\sup_{P\in\mathcal P}
\mathbb E_{P^n}\|\hat{\tau}-\tau\|_2^2
\gtrsim n^{-\frac{2}{2+d/\gamma}},
\quad
\inf_{\hat{\tau}}\sup_{P\in\mathcal P}
\mathbb E_{P^n}\|\hat{\tau}-\tau\|_\infty^2
\gtrsim
n^{-\frac{2}{2+d/\gamma}}
(\log n)^{\frac{2}{2+d/\gamma}}.
\]
Combining these bounds with Propositions~\ref{prop::L2} and \ref{prop::Linf} establishes Theorem~\ref{thm::lower}.

\section{Upper Bound} \label{sec::upper}
In this section, we show that the rates in Theorem~\ref{thm::lower} are attainable using the estimator proposed in \cite{kennedy2024minimax} with modifications to the kernel and second-order basis, under suitable regularity conditions. Following \cite{kennedy2024minimax}, we first introduce the estimator and derive error decompositions under the $L_2(\mathcal X)$ and $L_\infty(\mathcal X)$ risks in Section~\ref{sec::upper-1}. The resulting components are then bounded in the subsequent two subsections. Under the $L_2(\mathcal X)$ risk, a similar bias-variance decomposition applies, and the bounds established in \cite{kennedy2024minimax} continue to hold when the required regularity conditions hold uniformly over $\mathcal X$. Under the $L_\infty(\mathcal X)$ risk, we employ a discretization argument and control the resulting terms using tail bounds for U-statistics together with continuity properties of the kernel and basis functions.

\subsection{Proposed estimator and decomposition} \label{sec::upper-1}
Let $\rho:\mathbb{R}^d  \to \mathbb{R}^q$ denote the polynomial basis used in \cite{kennedy2024minimax}. Specifically, for each coordinate $x_j \ (j=1, \cdots, d)$, define $\rho(x_j) = \{\rho_0(x_j), \cdots, \rho_{\lfloor \gamma \rfloor }(x_j )\}^T$ as the first $(\lfloor \gamma \rfloor+1)$ terms of the Legendre polynomial series (shifted to be orthonormal on $[0,1]$), where
\[
\rho_m(x_j) = \sum_{\ell=0}^m (-1)^{\ell+m} \sqrt{ 2m+1} {m \choose \ell} {m+\ell \choose \ell} x_j^{\ell}.
\]
Define $\rho(x)$ as the collection of all tensor-product basis functions of the form
\[
\prod_{j=1}^d \rho_{\kappa_j} (x_j),
\]
where $(\kappa_1, \cdots, \kappa_d)$ satisfies $\sum_{j=1}^d \kappa_j \leq \lfloor \gamma \rfloor$. The resulting basis has dimension $q={d+\lfloor{\gamma} \rfloor \choose \lfloor{\gamma} \rfloor}$. Let $b_k :\mathbb{R}^d \to \mathbb{R}^{k}$ denote a second-order basis such that $\int b_k(v) b_k(v)^T dv$ has eigenvalues bounded above and below away from zero, and satisfying the additional properties specified in (\ref{eq::b-L2}-\ref{eq::b-row-sum}). For irregular covariate supports, the localized second-order basis and kernel may need to depend on the support geometry. Appendix~\ref{subsec:density_and_construction} provides such constructions and verifies the corresponding localized conditions directly. For simplicity, the main text uses a common reference basis and kernel, localized by the translation and rescaling below. For a generic point in the covariate space $x \in \mathcal{X}$, define the localized bases as
\begin{equation} \label{eq::kpho}
\rho_h^x(X_i) = \rho\!\left(\frac{X_i - x}{h} + \frac{1}{2}\right), \quad
b_{hk}^x(X_i)
= b_k\!\left(\frac{X_i - x}{h} + \frac{1}{2}\right).
\end{equation}
For this common localization, we use the following kernel $K_k$ and its localized version $K_{hk}^x$. Define the univariate bump function
\begin{equation} \label{eq::k-tilde}
\tilde{K}(t) = \frac{g(t)}{g(1-t)+g(t)}
\end{equation}
where $g(t) = \exp(-1/t) \mathbbm{1}\{t > 0\}$. Then $\tilde{K}(t)=0$ for $t \leq 0$, $\tilde{K}(t)=1$ for $t \geq 1$, and $\tilde{K}$ is monotonically increasing on $(0,1)$ with uniformly bounded derivative. For $u=(u_1, \cdots, u_d)$ and $k \geq 2$, define
\[
K_k(u) = \prod_{j=1}^d \tilde{K}(\frac{1/2-|u_j-1/2|}{1/2k^2}), \quad
K_{hk}^x(X_i) = \frac{1}{h^d} K_k( \frac{X_i -x}{h}+\frac{1}{2} ).
\]
Then, $K_{hk}^x$ is a continuous function taking values in $[0, h^{-d}]$ and satisfying
\begin{equation} \label{eq::K-cond1}
K_{hk}^x(X_i) =
\begin{cases}
    h^{-d}, & X_i \in \mathcal{C}_{h(1-1/k^2)}(x) \\
    0, & X_i \notin \mathcal{C}_{h}(x)
\end{cases}
\end{equation}
and, for any $x, x^\prime \in \mathcal{X}$,
\begin{equation} \label{eq::K-lip}
|K_{hk}^x(X_i) - K_{hk}^{x^\prime}(X_i)| \lesssim \frac{k^2}{h^{d+1}} \|x-x^\prime\|.
\end{equation}
The proposed estimator is then defined as
\begin{equation} \label{eq::estimator}
\hat{\tau}(x) = \rho(1/2)^T \hat{Q}_x^{-1}\hat{R}_x,
\end{equation}
where
\[
    \hat{Q}_x = \mathbb{P}_n \{
    \rho_h^x(X_1) K_{hk}^x(X_1) \hat{\varphi}_{a_1}(Z_1) \rho_h^x(X_1)^T
    \} +
    \mathbb{U}_n \{
    \rho_h^x(X_1) K_{hk}^x(X_1)
    \hat{\varphi}_{a_2}^x(Z_1, Z_2)
    K_{hk}^x(X_2)
    \rho_h^x(X_2)^T
    \},
\]
\[
    \hat{R}_x = \mathbb{P}_n \{
    \rho_h^x(X_1) K_{hk}^x(X_1) \hat{\varphi}_{y_1}(Z_1)
    \} +
    \mathbb{U}_n \{
    \rho_h^x(X_1) K_{hk}^x(X_1)
    \hat{\varphi}_{y_2}^x(Z_1, Z_2)
    K_{hk}^x(X_2)
    \},
\]
with the influence functions given by
\begin{align*}
    \hat{\varphi}_{a_1}(Z) & = A\{A-\hat{\pi}(X)\} \\
    \hat{\varphi}_{y_1}(Z) & = \{Y-\hat{\mu}_0(X)\}\{A-\hat{\pi}(X)\}\\
    \hat{\varphi}_{a_2}^x(Z_1, Z_2) & = -\{A_1 - \hat{\pi}(X_1) \}
    b_{hk}^x(X_1)^T\hat{\Omega}_x^{-1}b_{hk}^x(X_2)A_2\\
    \hat{\varphi}_{y_2}^x(Z_1, Z_2) & = -\{A_1 - \hat{\pi}(X_1) \}
    b_{hk}^x(X_1)^T\hat{\Omega}_x^{-1}b_{hk}^x(X_2)\{Y_2 - \hat{\mu}_0(X_2)\},
\end{align*}
where $\hat{\Omega}_x  = \int b_{hk}^x(X) K_{hk}^x(X) b_{hk}^x(X)^Td\hat{F}(X)$. The nuisance estimators $(\hat{F}, \hat{\mu}_0, \hat{\pi})$ are constructed from a separate training sample $D^n$ independent of the sample on which $\mathbb{P}_n$ and $\mathbb{U}_n$ are performed. The constructed estimator $\hat{\tau}(x)$ can be viewed as estimating
\[
\tau_h(x) = \rho(1/2)^T \theta_x,
\]
where $\theta_x$ denotes the coefficient vector of a locally weighted least-squares projection at $x$, defined as
\[
\theta_x=
{\arg\min}_{\beta} \mathbb{E}
\left[
K_{hk}^x(X)\pi(X)
\left(1-\pi(X)\right)
\left(\tau(X) - \beta^T\rho_h^x(X) \right)^2
\right] = Q_x^{-1} R_x
\]
with
\[
Q_x = \int \rho_h^x(X) K_{hk}^x(X)\pi(X)
\left(1-\pi(X)\right)\rho_h^x(X)^T dF(X)
\]
\[
R_x = \int \rho_h^x(X) K_{hk}^x(X)\pi(X)
\left(1-\pi(X)\right)\tau(X) dF(X).
\]
In $\hat{Q}_x$ and $\hat{R}_x$, second-order U-statistic terms are incorporated to further debias the first-order terms in the estimating equation $R_x - Q_x \theta_x =0$. This can be seen more clearly in the proof of Proposition 5 of \cite{kennedy2024minimax}. We refer to their Section~4.1 for additional discussion of the motivation and construction of the estimator.

For fixed $x \in \mathcal{X}$ and $v \in [0,1]^d$, define
\[
dF^{*,x}(v) = f(x+h(v-1/2))dv,
\]
whenever $F$ admits a density $f$, which defines a rescaled version of the covariate distribution in $\mathcal{C}_h(x)$ on $[0,1]^d$. For a function $g: [0,1]^d \rightarrow \mathbb{R}$, define its $K\cdot F^{*,x}$-weighted projection  onto $b_k$ as
\begin{align*}
\Pi_xg(v)
= \ &
b_k(v)^T  \arg \min_{\theta} \int_{[0,1]^d} \left(
g(u)-b_k(u)^T\theta
\right)^2K_{k}(u) dF^{*,x}(u) \\
= \ & b_k(v)^T \Omega_x^{-1} \int_{ [0,1]^d} b_k(u) g(u) K_{k}(u) dF^{*,x}(u),
\end{align*}
where
\[
\Omega_x
=
\int_{ [0,1]^d}K_{k}(u)b_k(u) b_k(u)^T dF^{*,x}(u).
\]
The corresponding estimated version is defined as
\[
\hat{\Pi}_xg(v)
= b_k(v)^T \hat{\Omega}_x^{-1} \int_{ [0,1]^d}  b_k(u) g(u) K_{k}(u) d{F}^{*,x}(u),
\]
where
\[
\hat{\Omega}_x =
\int_{ [0,1]^d} K_{k}(u) b_k(u) b_k(u)^T d\hat{F}^{*,x}(u).
\]

The kernel is designed to satisfy three requirements. First, $|K^x_{hk}(X_i)-K_{hk}^{x^\prime}(X_i)|$ must be controlled by $\|x-x^\prime\|$. The uniform kernel used in \cite{kennedy2024minimax} does not have this property because it is discontinuous at the boundary. Second, $K_k$ is pointwise bounded by the uniform kernel, so the corresponding upper bounds continue to apply. Third, $K_k$ differs from the uniform kernel only on a boundary strip of volume $1-(1-1/k^2)^d\asymp1/k^2$, which is $o(1/k)$ whenever $1/k=o(1)$. Together with the properties of $b_k$ in \eqref{eq::b-sup}, these facts imply that the smallest eigenvalues of $\Omega_x$ and $Q_x$ remain bounded away from zero under conditions analogous to those in Propositions~S5 and~S6 of \cite{kennedy2024minimax}. Sufficient conditions ensuring that the eigenvalues of $Q_x$, $\Omega_x$, and $\hat{\Omega}_x$ are uniformly bounded above and away from zero are discussed in Appendix~\ref{app::eigenvalues}.

For the second-order basis, \cite{kennedy2024minimax} assumes that $b_k(v)$ satisfies the optimal Hölder approximation property with respect to the kernel-weighted $L_2$ norm:
 \begin{equation} \label{eq::b-L2}
    \|(I- \Pi_x)g\|_{K\cdot F^{*,x}}\lesssim k^{-s/d} \text{ for any s-smooth function }g,
\end{equation}
where $\|\cdot\|_{K \cdot F^{*,x}}^2 = \int (\cdot)^2K_{k}(v)dF^{*,x}(v)$. As discussed in \cite{belloni2015some}, this condition holds for a broad class of basis functions under standard regularity conditions, including spline, CDV wavelet, and local polynomial partition series bases. The matching upper bound under the $L_2(\mathcal{X})$ risk can also be established under this condition. To establish the upper bound under the $L_\infty(\mathcal{X})$ norm, we additionally impose the uniform boundedness condition:
\begin{equation} \label{eq::b-sup}
    \sup_{v \in [0,1]^d} \| b_k(v) \|,
    \sup_{v \in [0,1]^d} \| b_k(v) \|_1 \lesssim \sqrt{k},
\end{equation}
and a local smoothness condition in the sense that there exists some constant $L>0$ such that
\begin{equation} \label{eq::b-lip}
    \|b_k(u)-b_k(v)\| \lesssim k^{L}\|u-v\|
\end{equation}
for all $u, v \in [0,1]^d$. Let $[b_k]_j(v)$ denote the $j$-th component of $b_k(v)$. For each coordinate $j=1, \cdots, k$, assume that
\begin{equation} \label{eq::b-support}
    \int_{[0,1]^d} [b_k]_j (v)^2 dv \lesssim 1,  \quad
    {\int_{[0,1]^d}|[b_k]_j(v)|\,dv \lesssim k^{-1/2}}, \quad
    \sup_{v \in [0,1]^d} \sum_{j=1}^k \mathbbm{1}\{ v \in  \operatorname{Supp}([b_k]_j )  \}\lesssim 1
\end{equation}
Thus, each component $[b_k]_j$ has bounded $L_2$ norm and $L_1$ norm of order $k^{-1/2}$, and only a uniformly bounded number of basis functions are nonzero at any fixed $v\in[0,1]^d$.

Finally, we impose the following additional regularity condition on $\Omega_x$ and $\hat{\Omega}_x$: for every $x \in \mathcal{X}$,
\begin{equation} \label{eq::b-row-sum}
\|\Omega_x^{-1}\|_{\infty},
\|\hat{\Omega}_x^{-1}\|_{\infty} \lesssim 1.
\end{equation}
For illustration, Appendix~\ref{subsec:full-support} verifies conditions \eqref{eq::b-sup}--\eqref{eq::b-row-sum} for a tensor-product normalized B-spline basis on $[0,1]^d$ with total dimension $k$. We expect analogous results hold for other locally supported bases, such as wavelets and local polynomial partition series. Henceforth, we assume that $b_k$ satisfies \eqref{eq::b-L2}--\eqref{eq::b-support}, as well as \eqref{eq::b-row-sum} whenever the eigenvalues of $\Omega_x$ and $\hat{\Omega}_x$ are bounded above and away from zero. The following bounds for the projection operators then hold.
\begin{proposition} \label{lemma::basis}
Assume that:
    \begin{itemize}
        \item[1.] For any $x \in \mathcal{X}$, the eigenvalues of $\Omega_x$ are bounded above and away from zero.
        \item[2.] $\|d\hat{F}/dF\|_{L_\infty(\mathcal{X})}$ and $\|dF/d\hat{F}\|_{L_\infty(\mathcal{X})}$ are bounded above.
        \item[3.] For every $x \in \mathcal{X}$, the density $dF$ is bounded above.
    \end{itemize}
Then for bounded function $g(v)$ on $[0,1]^d$,
    \[
    \sup_{x \in \mathcal{X}} \|\hat{\Pi}_x g\|_{K \cdot F^{*,x}}^2 \lesssim 1, \quad
    \sup_{x \in \mathcal{X}}\sup_{v \in [0,1]^d} |{\Pi}_x g(v)| \lesssim 1 , \quad
    \sup_{x \in \mathcal{X}}\sup_{v \in [0,1]^d} |\hat{\Pi}_x g(v)| \lesssim 1
    \]
\end{proposition}
\begin{proof}
See Appendix~\ref{app::basis}.
\end{proof}

Following Proposition~3 of \cite{kennedy2024minimax}, we next control the approximation error arising from replacing $\tau(x)$ with $\tau_h(x)$. This is formalized in the following proposition.

\begin{proposition}
For any fixed $x \in \mathcal{X}$ and $u \in \mathcal{C}_h(x)$, define the local estimating function
    \begin{equation} \label{eq::tilde-tau}
    \tilde{\tau}_h^x(u) =  \rho_h^x(u)^T \theta_x = \rho_h^x(u)^T Q_x^{-1} R_x.
    \end{equation}
In particular, $\tilde{\tau}_h^x(x) = \tau_h(x) = \rho(1/2)^TQ_x^{-1}R_x$. Assume that:
    \begin{itemize}
        \item[1.] $\tau(x)$ is $\gamma$-smooth.
        \item[2.] $\lambda_{\min}(Q_x)$ is uniformly bounded away from zero for any $x \in \mathcal{X}$.
        \item[3.] $\int \mathbbm{1}\{ X \in \mathcal{C}_h(x) \}dF(X) \lesssim h^d$ uniformly for $x \in \mathcal{X}$.
    \end{itemize}
Then
    \[
    \sup_{x \in \mathcal{X}}\sup_{u \in \mathcal{C}_h(x)} | \tilde{\tau}_h^x(u) - \tau(u)| \lesssim h^\gamma
    \]
and, in particular,
    \[
     \sup_{x \in \mathcal{X}} |{\tau}_h(x) - \tau(x)| \lesssim h^\gamma.
    \]
Consequently,
    \begin{equation} \label{eq::diff-tau-h}
    \| \tau_h - \tau \|^2_{L_2(\mathcal{X})} \lesssim h^{2\gamma}, \quad
    \|\tau_h - \tau\|^2_{L_\infty(\mathcal{X})} \lesssim h^{2\gamma}.
      \end{equation}
\end{proposition}
\begin{proof}
The result follows from the argument in Appendix~B.3 of \cite{kennedy2024minimax}. The uniform lower bound on $\lambda_{\min}(Q_x)$ and the uniform density condition make the argument uniform over $x\in\mathcal{X}$.
\end{proof}

Therefore, to control the estimation error of $\hat{\tau}(x)$ relative to $\tau(x)$, it suffices to analyze the error with respect to the surrogate target $\tau_h(x)$, up to an approximation remainder of order $h^{2\gamma}$. The next proposition first reduces this error to that of estimating $Q_x$ and $R_x$, and then derives corresponding decompositions under the $L_2(\mathcal{X})$ and $L_\infty(\mathcal{X})$ risks.

\begin{proposition} \label{prop::decomposition}
Define $S_x = R_x - Q_x(Q_x^{-1}R_x)$ and $\hat{S}_x = \hat{R}_x - \hat{Q}_x(Q_x^{-1}R_x)$. The estimator in (\ref{eq::estimator}) satisfies
\begin{equation} \label{eq::decompose-tau}
    |
    \hat{\tau}(x) - \tau_h(x)
    | =
    |\rho(1/2)^T \hat{Q}_x^{-1} \hat{S}_x|
    \leq \|\rho(1/2)\|
    \left(
    \|Q_x^{-1}\|
    + \|\hat{Q}_x^{-1}-Q_x^{-1}\|
    \right)
    \|\hat{S}_x - S_x\|.
\end{equation}
If $\|Q_x^{-1}\|$ and $\|\hat{Q}_x^{-1}-Q_x^{-1}\|$ are uniformly bounded over $x \in \mathcal{X}$, then the estimation error reduces to controlling $\|\hat{S}_x - S_x\|$. In particular,
\begin{equation} \label{eq::decom-l2}
    \mathbb{E}\|\hat{\tau}-\tau_h\|_{L_2(\mathcal{X})}^2
    \lesssim
    \sup_{x \in \mathcal{X}} \max_{1 \leq j \leq q}
    \mathbb{E}
    \left \{
    \mathbb{E}
    \left(
    [\hat{S}_x - S_x]_j \mid D^n
    \right)^2 +
    \operatorname{Var}([\hat{S}_x]_j \mid D^n )
    \right\}
\end{equation}
and
\begin{align}
    \nonumber
    \mathbb{E}&\|\hat{\tau}-\tau_h\|_{L_\infty(\mathcal{X})}^2  \lesssim
    \mathbb{E} \left\{
    \mathbb{E} \left(
    \max_{1\le \ell \le M}
    \max_{1 \le j \le q}
    \left| [\hat S_{x_\ell}]_j-\mathbb{E}([\hat S_{x_\ell} ]_j \mid D^n) \right|^2
    \mid D^n
    \right) \right. \\
    \label{eq::decom-l_infty}
    &
    \left.
    \quad +
     \mathbb{E}\left(
    \sup_{x, x^\prime \in \mathcal{X}: \|x-x'\|\lesssim M^{-1/d}}
    \|\hat S_x-\hat S_{x'}\|^2
    \mid D^n
    \right)
    +
    \sup_{x \in \mathcal{X}}
     \max_{1 \leq j \leq q}
     \left|
     \mathbb{E} \left( [\hat{S}_x - S_x]_j \mid D^n \right)
     \right|^2
    \right\}.
\end{align}
where $[ \cdot ]_j$ denotes the $j$-th coordinate of a vector, and $x_1,\ldots,x_M$ form a maximal packing of $\mathcal{X}$.
\end{proposition}

\begin{proof}
See Appendix~\ref{app::decomposition}.
\end{proof}
The $L_2(\mathcal{X})$ bound in (\ref{eq::decom-l2}) closely mirrors the bias-variance decomposition of \cite{kennedy2024minimax}, and each component can be bounded by analogous arguments once the regularity conditions hold uniformly over $x \in \mathcal{X}$. The $L_\infty(\mathcal{X})$ bound in (\ref{eq::decom-l_infty}) is more challenging and requires controlling three distinct components: (i) the maximal stochastic fluctuation over the packing set $\{\hat{S}_{x_k}\}_{k=1}^M$, (ii) the local oscillation of $\hat{S}_{x}$ between sufficiently close points, and (iii) the bias term. While the bias term can be handled by the same arguments as in (\ref{eq::decom-l2}), bounding the other two terms requires additional techniques. Concentration inequalities for U-statistics are used to control the maximal fluctuation term, and a continuity argument, made possible by the modified kernel and basis construction, is applied to bound the local oscillation in $\|\hat{S}_x - \hat{S}_{x'}\|$ under the constraint $\|x-x'\| \lesssim M^{-1/d}$. The next two subsections bound the terms in \eqref{eq::decom-l2} and \eqref{eq::decom-l_infty} and derive the resulting $L_2(\mathcal{X})$ and $L_\infty(\mathcal{X})$ convergence rates.

\subsection{$L_2(\mathcal{X})$ risk} \label{sec::upper-2}
The next three propositions bound the bias and variance terms appearing in (\ref{eq::decom-l2}), following arguments similar to Propositions 5-7 of \cite{kennedy2024minimax}.
\begin{proposition} \label{prop::bias}
Assume:
    \begin{itemize}
        \item[1.] $\lambda_{\max}(\Omega_x)$ is bounded above uniformly for $x \in \mathcal{X}$
        \item[2.] $\hat{\pi}(x)-\pi(x)$ is $\alpha$-smooth
        \item[3.] $\hat{\mu}_0(x) - \mu_0(x)$ is $\beta$-smooth
    \end{itemize}
Then for any $x \in \mathcal{X}$
    \begin{align*}
       \max_{1 \leq j \leq q}\left| \mathbb{E}\left( [\hat{S}_x - S_x]_j \mid D^n\right) \right|
       \lesssim  & (h/k^{1/d})^{2s} + h^\gamma (h/k^{1/d})^{\alpha} \\
       & +
       \left( h^\gamma + \|\hat{\mu}_0 - \mu_0\|_{K \cdot F^{*,x}} \right)
       \left( \|\hat{\pi}-\pi\|_{K \cdot F^{*,x}} \|\hat{\Omega}_x^{-1} - \Omega^{-1}_x\| \right).
    \end{align*}
\end{proposition}
\begin{proof}
The proof follows the same logic as Proposition 5 in \cite{kennedy2024minimax}. In particular, the first two terms are the main bias terms that remain even when the covariate distribution is known. The first term, $(h/k^{1/d})^{2s}$, arises from the product of the projection approximation errors of the nuisance estimation errors $\hat\pi-\pi$ and $\hat\mu_0-\mu_0$, which are of orders $(h/k^{1/d})^\alpha$ and $(h/k^{1/d})^\beta$, respectively. The second term, $h^\gamma(h/k^{1/d})^\alpha$, is the product of the local-polynomial approximation remainder for the CATE and the projection approximation error of $\hat\pi-\pi$.
\end{proof}

By Lemma S2 of \cite{kennedy2024minimax}, if the eigenvalues of $\Omega_x$ are bounded above and away from zero over $x \in \mathcal{X}$, and if both $\|d \hat{F}^{*,x}/dF^{*,x}\|_{L_\infty(F^{*,x})}$ and $\|d F^{*,x}/d\hat{F}^{*,x}\|_{L_\infty(F^{*,x})}$ are bounded uniformly over $x\in\mathcal X$, then
\[
\|\hat{\Omega}_x^{-1}-\Omega_x^{-1}\|
\lesssim
\left\|
\frac{d\hat F^{*,x}}{dF^{*,x}}-1
\right\|_{L_\infty(F^{*,x})}.
\]
Substituting this bound into Proposition~\ref{prop::bias} yields the following refinement.
\begin{proposition} \label{prop::bias-refined}
Under the assumption of Proposition~\ref{prop::bias}, additionally assume
    \begin{itemize}
        \item[1.] $\lambda_{\min}(\Omega_x)$ is bounded below away from zero uniformly for $x \in \mathcal{X}$
        \item[2.] $\|d\hat{F}/dF\|_{L_\infty(\mathcal{X})}$ and  $\|dF/d\hat{F}\|_{L_\infty(\mathcal{X})}$ are bounded above
        \item[3.] For some nonnegative deterministic sequence
$r_n=o(h^\gamma)$, uniformly over $x\in\mathcal X$,
        \[
        \left\|\frac{d\hat F^{*,x}}{dF^{*,x}}-1\right\|_{L_\infty(F^{*,x})}
        \lesssim
        \frac{(h/k^{1/d})^{2s}+r_n}
        {\|\hat\pi-\pi\|_{K\cdot F^{*,x}}
        \bigl(\|\hat\mu_0-\mu_0\|_{K\cdot F^{*,x}}+h^\gamma\bigr)}.
        \]
    \end{itemize}
Then, for sufficiently large $n$ with $h/k^{1/d}\leq1$,
    \[
    \sup_{x \in \mathcal{X}}  \max_{1\leq j \leq q}\mathbb{E}\left( [\hat{S}_x - S_x]_j \mid D^n\right)^2
       {\lesssim (h/k^{1/d})^{4s}+h^{2\gamma}.}
    \]
\end{proposition}
Here, we impose the slightly stronger condition $r_n=o(h^\gamma)$ rather than $O(h^\gamma)$ to establish the stability of $\hat Q_x^{-1}$. Further discussion can be found after Theorem~\ref{thm::upper}.

For the variance term, Proposition~7 of \cite{kennedy2024minimax} can be adapted to obtain the following bound.

\begin{proposition} \label{prop::var}
Assume:
    \begin{itemize}
        \item[1.] $\int \mathbbm{1}\{X \in \mathcal{C}_h(x) \} dF(X) \lesssim h^d$ for any $x \in \mathcal{X}$
        \item[2.] $y^2$, $\hat{\pi}^2$, $\hat{\mu}_0^2$  are all bounded above
        \item[3.] $\lambda_{\max}(\Omega_x)$ is bounded above for any $x \in \mathcal{X}$
    \end{itemize}
Then
    \[
   \sup_{x \in \mathcal{X}} \max_{1 \leq j \leq q} \operatorname{Var}([\hat{S}_x]_j \mid D^n)
    \lesssim \frac{1}{nh^d} \left(
    1+\frac{k}{nh^d} (1+ \sup_{x \in \mathcal{X}} \|\hat{\Omega}^{-1}_x - \Omega^{-1}_x\|^2)
    \right)
    \]
\end{proposition}

Combining (\ref{eq::diff-tau-h}) and (\ref{eq::decom-l2}) with the bounds of Proposition~\ref{prop::bias-refined} and~\ref{prop::var}, and assuming all of their conditions hold, we obtain
\[
\mathbb{E}_{P^n}\|\hat{\tau} - \tau \|_{L_2(\mathcal{X})}^2 \lesssim h^{2\gamma} + (h/k^{1/d})^{4s}+
{\frac{1}{nh^d}(1+\frac{k}{nh^d})},
\]
which coincides with the squared version of the bound derived in \cite{kennedy2024minimax}. In the low-smoothness regime $s < (d/4)/(1+d/2\gamma)$, choosing
\[
h \asymp n^{-\frac{1/\gamma}{d/2\gamma+d/4s+1}}, \quad
k \asymp n^{\frac{(d/2s-d/\gamma)}{d/2\gamma+d/4s+1}}
\]
yields the overall rate
\[
\mathbb{E}_{P^n} \|\hat{\tau}-\tau\|_{L_2(\mathcal{X})}^2
\lesssim n^{-\frac{2}{d/2\gamma+d/4s+1}}.
\]
On the other hand, in the high smoothness regime $s > (d/4)/(1+d/2\gamma)$, taking
\[
h \asymp n^{-\frac{1}{2\gamma+d}}, \quad
k \asymp nh^d
\]
gives
\[
\mathbb{E}_{P^n} \|\hat{\tau}-\tau\|_{L_2(\mathcal{X})}^2
\lesssim n^{-\frac{2}{2+d/\gamma}}.
\]
The resulting bounds under the $L_2(\mathcal{X})$ risk  are summarized in Theorem~\ref{thm::upper}.

\subsection{$L_\infty(\mathcal{X})$ risk} \label{sec::upper-3}
We exploit the structure of $[\hat{S}_x]_j$ to control the maximal stochastic fluctuation over the packing set $\{x_\ell\}_{\ell=1}^M \subset \mathcal{X}$, namely,
\[
\mathbb{E}\left(
\max_{1 \leq \ell \leq M} \max_{1 \leq j \leq q}
\left|
[\hat{S}_{x_\ell}]_j - \mathbb{E}([\hat{S}_{x_\ell}]_j \mid D^n)
\right|^2 \mid D^n
\right).
\]
For any fixed $x$, $[\hat{S}_x]_j$ can be represented as a second-order U-statistic with a non-degenerate kernel. This representation allows us to apply concentration inequalities for U-statistics, together with a maximal inequality over the packing set, to derive a bound for the maximal fluctuation. The resulting proposition is stated below, and the full proof is deferred to Appendix~\ref{app::l_inf_1}.

\begin{proposition} \label{prop::l_inf_1}
Under the conditions of Proposition~\ref{lemma::basis} and Proposition~\ref{prop::bias-refined}, suppose
    \begin{itemize}
        \item[1.] $y^2$, $\hat{\mu}_0^2$, and $\hat{\pi}^2$ are bounded above,
        \item[2.] $n$, $h$, and $k$ satisfy $nh^d \gtrsim (\log n)^2$, $k\gtrsim\log n$, and $k (\log n)^3 \lesssim (nh^d)^2$,
        \item[3.] $M$ is a polynomial in $n$ of fixed degree.
    \end{itemize}
Then, the first term in (\ref{eq::decom-l_infty}) can be bounded as
    \begin{align*}
     \mathbb{E}\left(
    \max_{1 \leq \ell \leq M} \max_{1 \leq j \leq q}
    \left|
    [\hat{S}_{x_\ell}]_j - \mathbb{E}([\hat{S}_{x_\ell}]_j \mid D^n)
    \right|^2 \mid D^n
    \right)
    \lesssim
    \frac{k \log n}{n^2 h^{2d}} + \frac{\log n}{nh^d}.
    \end{align*}
\end{proposition}

The next proposition controls the local oscillation of $\hat{S}_{x}$ between sufficiently nearby points. The key observation is that the continuity of the kernel $K_k$ and the basis functions $\rho$ and $b_k$ allows $\| \hat{S}_{x} - \hat{S}_{x^\prime} \|$ to be controlled whenever $\|x - x^\prime\|$ is sufficiently small. In particular, under the continuity conditions (\ref{eq::K-lip}) and (\ref{eq::b-lip}), the local oscillation can be uniformly bounded by $k/n^2h^{2d}$ over pairs $x, x'\in \mathcal{X}$ satisfying $  \|x-x'\| \lesssim M^{-1/d}$, provided that $M^{1/d} \gtrsim \frac{n k^{L+7/2}}{h^{2d+1}}$. This is formally stated in the proposition below, with the full proof deferred to Appendix~\ref{app::l_inf_2}.
\begin{proposition} \label{prop::l_inf_2}
Assume that
    \begin{itemize}
        \item[1.] $\lambda_{\min}(Q_x)$ and $\lambda_{\min}({\Omega}_x)$ are uniformly bounded away from zero for $x \in \mathcal{X}$
        \item[2.]  $y$, $\hat{\pi}$, $\hat{\mu}_0$, and $\|dF/d\hat{F}\|_{L_\infty(\mathcal{X})}$ are bounded above
        \item[3.]  $M^{1/d} \gtrsim  \frac{n k^{L+7/2}}{h^{2d+1}}$
    \end{itemize}
Then
    \[
     \mathbb{E}\left(
    \sup_{\|x-x'\|\lesssim \,M^{-1/d}}
    \|(\hat S_x-\hat S_{x'})\|^2
    \mid D^n
    \right) \lesssim \frac{k}{n^2h^{2d}}
    \]
\end{proposition}

When the tuning parameters $h$ and $k$ are chosen to be polynomial in $n$ with fixed exponents, one can select $M$ such that $M^{1/d} \gtrsim  \frac{n k^{L+7/2}}{h^{2d+1}}$, while ensuring that $M$ remains bounded by a polynomial in $n$ of fixed degree. Consequently, the conditions of both Proposition~\ref{prop::l_inf_1} and Proposition~\ref{prop::l_inf_2} are satisfied. Combining (\ref{eq::diff-tau-h}) and (\ref{eq::decom-l_infty}) with the bias bound from Proposition~\ref{prop::bias-refined}, the fluctuation bound from Proposition~\ref{prop::l_inf_1}, and the local oscillation bound from Proposition~\ref{prop::l_inf_2}, we obtain
\[
\mathbb{E}_{P^n} \|\hat{\tau}-\tau\|_{L_\infty(\mathcal{X})}^2
\lesssim h^{2\gamma} + (h/k^{1/d})^{4s} + \frac{k \log n}{n^2 h^{2d}} + \frac{\log n}{nh^d} .
\]
In the low-smoothness regime, when $s < \frac{d/4}{1+d/(2\gamma)} (1-\frac{\log\log n}{\log n})$, balancing $h^{2\gamma}$, $(h/k^{1/d})^{4s}$, and $k\log n/(n^2h^{2d})$ by choosing
\begin{equation} \label{eq::rate}
h \asymp
n^{-\frac{1/\gamma}{1+d/(2\gamma)+d/(4s)}}
(\log n)^{\frac{1/(2\gamma)}
{1+d/(2\gamma)+d/(4s)}},
\quad
(h/k^{1/d})^{4s} \asymp h^{2\gamma}
\end{equation}
yields
\[
\mathbb{E}_{P^n}
\|\hat{\tau}-\tau\|_{L_\infty(\mathcal{X})}^2
\lesssim
n^{-\frac{2}{1+d/(2\gamma)+d/(4s)}}
(\log n)^{\frac{1}{1+d/(2\gamma)+d/(4s)}}.
\]
Under these choices, the condition defining the low-smoothness regime is equivalent to $k\gtrsim nh^d$. Consequently, the first-order fluctuation term $\log n/(nh^d)$ is dominated by the second-order fluctuation term $k\log n/(n^2h^{2d})$. The two non-stochastic terms $h^{2\gamma}$ and $(h/k^{1/d})^{4s}$ and the second-order fluctuation term are balanced at the same order and jointly determine the rate, with the logarithmic factor entering through the second-order fluctuation.

In the high-smoothness regime, when $s \geq \frac{d/4}{1+d/(2\gamma)} (1-\frac{\log\log n}{\log n})$, taking
\[
h \asymp
\left(\frac{\log n}{n}\right)^{\frac{1}{2\gamma+d}},
\quad
k \asymp nh^d
\]
yields
\[
\mathbb{E}_{P^n}
\|\hat{\tau}-\tau\|_{L_\infty(\mathcal{X})}^2
\lesssim
\left(\frac{\log n}{n}\right)^{\frac{2}{2+d/\gamma}}.
\]
In this regime, the first- and second-order stochastic terms have the same order, while the nuisance-induced term $(h/k^{1/d})^{4s}$ is no larger than $h^{2\gamma}$ and therefore does not determine the rate. The final rate is obtained by balancing the common stochastic order $\log n/(nh^d)$ with $h^{2\gamma}$, as in uniform nonparametric estimation. The conditions on $n$, $h$, and $k$ required by Proposition~\ref{prop::l_inf_1} are automatically satisfied in both regimes. The resulting upper bounds are summarized in the next subsection.

\subsection{Main results and discussion} \label{sec::upper-4}
The following theorem summarizes the upper bounds for the proposed estimator under the $L_2(\mathcal{X})$ and $L_\infty(\mathcal{X})$ risks.

\begin{theorem} \label{thm::upper}
Suppose the following regularity conditions hold:
    \begin{itemize}
        \item[1.] For any $x \in \mathcal{X}$, the eigenvalues of $Q_x$ and $\Omega_x$ are bounded above and away from zero, and the covariate density is bounded above and satisfies $\int \mathbbm{1}\{X \in \mathcal{C}_h(x)\} dF(X) \lesssim h^d$.
        \item[2.] $\hat{\pi}-\pi$ is $\alpha$-smooth and $\hat{\mu}_0-\mu_0$ is $\beta$-smooth.
        \item[3.] $y^2$, $\hat{\pi}^2$, $\hat{\mu}_0^2$, $\| \hat{Q}_x^{-1}-Q_x^{-1} \|_{L_\infty(\mathcal{X})}$, $\|d\hat{F}/dF\|_{L_\infty(\mathcal{X})}$, and $\|d{F}/d\hat{F}\|_{L_\infty(\mathcal{X})}$ are bounded above.
    \end{itemize}
In addition, assume
    \begin{itemize}
        \item[1.] $\tau(x)$ is $\gamma$-smooth,
$\mu_0(x)$ is $\beta$-smooth, $\pi(x)$ is $\alpha$-smooth, and $\varepsilon \leq \pi(x) \leq 1-\varepsilon$ for some constant $\varepsilon >0$.
        \item[2.]
The covariate distribution $F$ is either known, in which case we take $\hat F=F$, or estimated sufficiently accurately such that uniformly over $x\in\mathcal X$,
        \[
        \left\|
        \frac{d\hat F^{*,x}}{dF^{*,x}}-1
        \right\|_{L_\infty(F^{*,x})}
        \lesssim
        \frac{
        (h/k^{1/d})^{2s}+r_n
        }{
        \|\hat\pi-\pi\|_{K\cdot F^{*,x}}
        \left(
        \|\hat\mu_0-\mu_0\|_{K\cdot F^{*,x}}+h^\gamma
        \right)
        }.
        \]
for some nonnegative deterministic sequence $r_n=o(h^\gamma)$, with $h$ and $k$ chosen as in Sections~\ref{sec::upper-2} and~\ref{sec::upper-3} for the respective risks.
        \end{itemize}
The localized basis and kernel satisfy the conditions of Section~\ref{sec::upper-1}, or their support-adapted counterparts verified in Appendix~\ref{subsec:density_and_construction}. The estimator defined in (\ref{eq::estimator}) then satisfies
        \[
        \mathbb{E}_{P^n}\|\hat{\tau} - \tau\|_{L_2(\mathcal{X})}^2
         \lesssim
         \begin{cases}
             n^{-\frac{2}{d/2\gamma+d/4s+1}} & s < \frac{d/4}{1+d/2\gamma}\\
             n^{-\frac{2}{2+d/\gamma}}& s \geq \frac{d/4}{1+d/2\gamma}\\
         \end{cases},
        \]
and
          \[
       \mathbb{E}_{P^n}\|\hat{\tau} - \tau\|_{L_\infty(\mathcal{X})}^2
         \lesssim
         \begin{cases}
             {n}^{-\frac{2}{d/2\gamma+d/4s+1}}
             ({\log n})^{\frac{1}{d/2\gamma+d/4s+1}}
             &
             s < \frac{d/4}{1+d/2\gamma}(1-\frac{\log \log n}{\log n})\\
             (\frac{\log n}{n})^{\frac{2}{2+d/\gamma}} &
             s \geq \frac{d/4}{1+d/2\gamma}(1-\frac{\log \log n}{\log n })\\
         \end{cases}.
        \]
\end{theorem}

Under the $L_2(\mathcal X)$ risk, integration over the covariate domain preserves the same bias--variance balance and transition threshold as in pointwise estimation. Under the $L_\infty(\mathcal X)$ risk, uniform control over the covariate domain introduces an additional $\log n$ factor into each stochastic term. In the low-smoothness regime, $k\gtrsim nh^d$, so the second-order fluctuation dominates the first-order fluctuation. The localization term $h^{2\gamma}$, the nuisance-induced term $(h/k^{1/d})^{4s}$, and the second-order fluctuation are then balanced at the same order, producing the nuisance-dependent logarithmic exponent in the low-smoothness rate. In the high-smoothness regime, $k\asymp nh^d$, so the first- and second-order stochastic terms have the same order, while the nuisance-induced term is no larger than the other leading terms and does not affect the resulting rate. Balancing the common stochastic order $\log n/(nh^d)$ with the localization term $h^{2\gamma}$ recovers the usual uniform nonparametric rate. The boundary between the two regimes occurs when $k\asymp nh^d$. Substituting the selected choices of $h$ and $k$ into this condition gives the adjusted transition threshold under the $L_\infty(\mathcal X)$ risk, with the factor $1-\log\log n/\log n$ arising from the additional $\log n$ factor in the stochastic bounds required for uniform control.

We next discuss how the main regularity conditions in Theorem~\ref{thm::upper} can be verified and interpreted. As discussed in Remark~10 of \cite{kennedy2024minimax}, the smoothness conditions on $\hat\pi-\pi$ and $\hat\mu_0-\mu_0$ hold when $\hat\pi$ and $\hat\mu_0$ have the same H\"older smoothness as $\pi$ and $\mu_0$, respectively. Such nuisance estimators can be obtained, for example, using suitably tuned local polynomial, spline, or wavelet methods. We next clarify the uniform inverse-stability condition imposed on $\hat Q_x$ in Theorem~\ref{thm::upper}. The theorem assumes $\sup_{x\in\mathcal X} \|\hat Q_x^{-1}-Q_x^{-1}\|\lesssim1$, which is a uniform analogue of the pointwise condition in \cite{kennedy2024minimax}. Strictly speaking, $\hat Q_x$ is random and may be singular, for example, when all treatment indicators in the estimation sample are zero. Let $\epsilon>0$ be a fixed numerical threshold, and suppose that $\inf_{x\in\mathcal X}\lambda_{\min}(Q_x)\geq2\epsilon$. Under the remaining conditions and suitable fixed tuning constants, Appendix~\ref{app::Q-stability} establishes
\begin{equation}\label{eq::Q-stability-tail}
\Pr\!\left\{
\inf_{x\in\mathcal X}\sigma_{\min}(\hat Q_x)<\epsilon
\right\}
\lesssim
\exp\!\left(-\frac{nh^d}{C_\epsilon\sqrt{k}}\right)
\end{equation}
for sufficiently large $n$. We use the smallest singular value here because $\hat Q_x$ need not be symmetric. On the complementary event,
\[
\begin{aligned}
\sup_{x\in\mathcal X}
\|\hat Q_x^{-1}-Q_x^{-1}\|
&\leq
\sup_{x\in\mathcal X}\|\hat Q_x^{-1}\|
+\sup_{x\in\mathcal X}\|Q_x^{-1}\|\\
&\leq
\frac{1}{\epsilon}+\frac{1}{2\epsilon}
=
\frac{3}{2\epsilon}.
\end{aligned}
\]
Accordingly, we define the estimator to be identically zero whenever $\inf_{x\in\mathcal X}\sigma_{\min}(\hat Q_x)<\epsilon$. In numerical implementations, $\epsilon$ may be chosen as a fixed stability tolerance above machine precision. Since $\tau$ is uniformly bounded, this event contributes at most a constant multiple of the probability bound in \eqref{eq::Q-stability-tail} to either squared risk. Under the stated tuning choices, this contribution decays faster than every polynomial rate and therefore leaves the rates in Theorem~\ref{thm::upper} unchanged.

Proposition~\ref{prop:eigen} provides sufficient conditions for the uniform eigenvalue bounds on $Q_x$ and $\Omega_x$, emphasizing the joint roles of the covariate distribution and the second-order basis. The corresponding bounds for $\hat\Omega_x$ follow from the density-ratio bounds. Subsection~\ref{subsec:density_and_construction} verifies these conditions and the second-order basis requirements for a family of covariate-density and support classes. In particular, the construction applies when the support is $[0,1]^d$, or more generally a single axis-aligned cube contained in $[0,1]^d$, and the density is bounded above and away from zero on its support. It also covers certain disconnected unions of cubes, including the local supports used in the lower-bound construction. These constructions also illustrate how the second-order basis can be adapted to the geometry of the covariate support.

When $F$ is known, we take $\hat F=F$, in which case the density-ratio and density-accuracy conditions hold automatically. Consequently, for a density bounded above and away from zero on $[0,1]^d$, or on a single axis-aligned cube, no smoothness condition on $f$ is required. When $F$ is unknown and must be estimated, additional regularity of the density gives a simple sufficient condition. For illustration, suppose that $f$ is supported on $[0,1]^d$, is $\zeta$-H\"older smooth, and is bounded above and away from zero. Suppose that the estimators satisfy
\begin{align*}
\|\hat f-f\|_{L_\infty([0,1]^d)}
\lesssim
\left(\frac{\log n}{n}\right)^{\zeta/(2\zeta+d)}, & \quad
\|\hat\pi-\pi\|_{L_\infty([0,1]^d)} \lesssim
\left(\frac{\log n}{n}\right)^{\alpha/(2\alpha+d)},\\
\|\hat\mu_0-\mu_0\|_{L_\infty([0,1]^d)}
&\lesssim
\left(\frac{\log n}{n}\right)^{\beta/(2\beta+d)},
\end{align*}
as can be achieved by standard nonparametric estimators under the stated smoothness conditions. Then, uniformly over $x\in\mathcal X$,
\begin{align*}
&\left\|
\frac{d\hat F^{*,x}}{dF^{*,x}}-1
\right\|_{L_\infty(F^{*,x})}
\|\hat\pi-\pi\|_{K\cdot F^{*,x}}
\bigl(
\|\hat\mu_0-\mu_0\|_{K\cdot F^{*,x}}
+h^\gamma
\bigr)\\
&\quad\lesssim
\left(\frac{\log n}{n}\right)^{
\zeta/(2\zeta+d)+\alpha/(2\alpha+d)}
\left\{
\left(\frac{\log n}{n}\right)^{\beta/(2\beta+d)}
+h^\gamma
\right\}.
\end{align*}
As in Remark~11 of \cite{kennedy2024minimax}, a sufficient smoothness condition for both losses is
\[
\frac{\zeta}{2\zeta+d}
>
M_{\alpha,\beta,\gamma,d},
\]
where $M_{\alpha,\beta,\gamma,d} = \frac{1}{ 1+d/(2\gamma)+ \{d/(4s)\vee(1+d/(2\gamma))\} } -\frac{1}{2+d/\alpha} -\frac{1}{2+d/\beta}$. Under the stated error bounds and tuning choices, the density-accuracy condition holds with
\[
r_n=
\left(\frac{\log n}{n}\right)^{
\zeta/(2\zeta+d)+\alpha/(2\alpha+d)}
\left\{
\left(\frac{\log n}{n}\right)^{\beta/(2\beta+d)}
+h^\gamma
\right\}
=o(h^\gamma).
\]

Thus, when $F$ is known, the upper bounds cover standard density classes on regular cube supports without requiring density smoothness. When $F$ must be estimated, additional assumptions, such as sufficient smoothness of its density, are needed to verify the density-accuracy condition in Theorem~\ref{thm::upper}. The problem may become more delicate when the support is unknown, particularly if it is irregular or disconnected. Extending the analysis to such settings remains an important direction for future research.

\section{Conclusion}
In this work, we established matching global minimax lower and upper bounds for CATE estimation under the $L_2(\mathcal X)$ and $L_\infty(\mathcal X)$ risks over the stated covariate-density and support classes, subject to the regularity conditions required for the upper bounds. For the lower bounds, we developed a many-fuzzy-hypotheses method that replaces each fixed alternative in a multiple-testing construction with a mixture over nuisance perturbations while keeping its CATE unchanged. A localized product-mixture $\chi^2$ bound controls the divergence of each resulting mixture from a common baseline, allowing a common framework to establish the lower bounds under both global losses. This strategy may also be useful for other nonparametric and semiparametric minimax problems involving nuisance functions. For the upper bounds, we modified the local polynomial R-learner of \cite{kennedy2024minimax} by replacing the uniform kernel with a continuous kernel and imposing additional regularity conditions on the second-order basis. Under the stated conditions, the resulting upper bounds match the lower bounds.

Our results show that the elbow phenomenon is not merely an artifact of pointwise estimation, but persists when the entire CATE function is estimated under global losses, thereby extending the regression--functional hybrid interpretation of \cite{kennedy2024minimax} to the global setting. In the low-smoothness regime, the rate depends on the smoothness of the nuisance functions, whereas in the high-smoothness regime it reduces to the usual rate for estimating a smooth regression function. Under the $L_2(\mathcal X)$ risk, averaging over the covariate domain leaves the pointwise transition threshold unchanged. Under the $L_\infty(\mathcal X)$ risk, controlling the error simultaneously over the covariate domain introduces additional logarithmic factors and shifts the threshold by the factor $1-\log\log n/\log n$. A similar distinction may also arise in other nonparametric problems where estimating the target function requires estimating nuisance functions.

The sufficient conditions for the upper bounds also illustrate how the covariate distribution enters the analysis through both the support geometry underlying the basis construction and the accuracy required when the distribution is estimated. Extending the analysis to broader covariate-density classes, particularly when the support is unknown, irregular, or disconnected, remains an important direction for future work.

\newpage
\appendix
\section{Proofs for Lower bound}
\subsection{Proof for Theorem~\ref{lemma::chi}} \label{app::chi}
\begin{proof}
First, the $\chi^2$-divergence admits a partition decomposition analogous to the decomposition established in Lemma 5.1 of \cite{robins2009semiparametric}. Consequently, the same argument yields
\begin{equation} \label{eq::chi-decompose}
\chi^2(
\int Q_\lambda^n d \pi (\lambda),
\int P_\lambda^n d \pi (\lambda)
) + 1 =
\mathbb{E}( \prod_{ {j\in S}}  (\chi_j^2(N_j)+1) )
\end{equation}
where $(N_1, \cdots, N_k)$ follows a multinomial distribution with parameters $n$ and  $(p_1,\cdots, p_k)$, and for $j \in S$,
\[
\chi_j^2(N_j) =
\begin{cases}
\chi^2(
\int Q_{j,\lambda_j}^{N_j} d \pi_j(\lambda_j),
\int P_{j,\lambda_j}^{N_j} d \pi_j(\lambda_j)
) &  N_j \geq 1 \\
0 & N_j = 0
\end{cases},
\]
and $\chi_j^2(N_j)=0$ for $j\notin S$.

Next, the bound of Lemma 5.2 in \cite{robins2009semiparametric} continues to hold in the present setting. Indeed, the quantity appearing in their equation (5.2) can be rewritten as
\begin{align*}
\mathbb{E}_{\bar{P}_n}\left(
\frac{
\left( \int [\prod_{i=1}^n q_\lambda(X_i) - \prod_{i=1}^n p_\lambda(X_i)] d\pi(\lambda) \right)^2}
{ \left( \int \prod_{i=1}^n p_\lambda(X_i) d\pi(\lambda) \right)^2}
\right)
= &
\mathbb{E}_{\bar{P}_n}\left(
\left(
\frac{\int \prod_{i=1}^n q_\lambda(X_i) d\pi(\lambda)}
{\int \prod_{i=1}^n p_\lambda(X_i) d\pi(\lambda)}
\right)^2
\right) - 1 \\
= &
\chi^2(
\int Q_\lambda^n d \pi (\lambda),
\int P_\lambda^n d \pi (\lambda)
)
\end{align*}
where $\bar{P}_n = \int P_\lambda^n  d\pi(\lambda)$. Therefore, for $j\in S$,
\[
\chi^2_j(N_j) \leq
2 \left(
\sum_{r=2}^{N_j}  \binom{N_j}{r} \delta_2^r + 2N_j^2 \sum_{r=1}^{N_j-1} \binom{N_j-1}{r}\delta_1^r \delta_2 + 2N_j^2 \varepsilon_1^{N_j-1}\delta_3
\right).
\]
Substituting this bound into (\ref{eq::chi-decompose}) and expanding the resulting product yields
\begin{align}
\nonumber
\chi^2(
\int Q_\lambda^n d \pi (\lambda),
\int P_\lambda^n d \pi (\lambda)
)
\leq &
\mathbb{E}
\left(
\prod_{{j\in S}}  \left(
2 \sum_{r=2}^{N_j}  \binom{N_j}{r} \delta_2^r + 4N_j^2 \sum_{r=1}^{N_j-1} \binom{N_j-1}{r}\delta_1^r \delta_2 +
4N_j^2 \varepsilon_1^{N_j-1}\delta_3 + 1\right) -1
\right) \\
\label{eq::bracket}
= &
\sum_{ {\emptyset\ne J\subseteq S} }
\mathbb{E}
\left(
\prod_{j \in J} \left(
2 \sum_{r=2}^{N_j}  \binom{N_j}{r} \delta_2^r + 4N_j^2 \sum_{r=1}^{N_j-1} \binom{N_j-1}{r}\delta_1^r \delta_2 +
4N_j^2 \varepsilon_1^{N_j-1}\delta_3
\right)
\right)
\end{align}
{where we used $\prod_{j\in S}(x_j+1)-1 = \sum_{\emptyset\ne J\subseteq S} \prod_{j \in J} x_j$ together with Tonelli's theorem.}

Now consider
\[
(N_1, \cdots, N_k) \sim \operatorname{Multinomial}(n; p_1, \cdots, p_k).
\]
By Property P2 of \cite{joag1983negative}, multinomial counts are negatively associated. Therefore, for any collection of nonnegative nondecreasing functions $f_1, \cdots, f_k$,
\begin{equation} \label{eq::NA}
\mathbb{E}\left( \prod_{j \in J} f_j(N_j) \right) \leq \prod_{j \in J}
\mathbb{E}\left( f_j(N_j) \right).
\end{equation}
Without loss of generality, assume $\varepsilon_1 \geq 1$, since otherwise $\varepsilon_1$ can be replaced by $1$. {For each $j\in S$, the quantity inside the product in (\ref{eq::bracket}) is a nonnegative nondecreasing function of $N_j$.} Applying (\ref{eq::NA}) therefore yields
\begin{align*}
\chi^2 (
\int Q_\lambda^n d \pi (\lambda),
\int P_\lambda^n d \pi (\lambda)
)
\leq & \sum_{{\emptyset\ne J\subseteq S}}  \prod_{j \in J}
\mathbb{E} \left(
2 \sum_{r=2}^{N_j}  \binom{N_j}{r} \delta_2^r + 4N_j^2 \sum_{r=1}^{N_j-1} \binom{N_j-1}{r}\delta_1^r \delta_2 +
4N_j^2 \varepsilon_1^{N_j-1}\delta_3
\right) \\
= &   \prod_{{j\in S}} \mathbb{E} \left(
2 \sum_{r=2}^{N_j}  \binom{N_j}{r} \delta_2^r + 4N_j^2 \sum_{r=1}^{N_j-1} \binom{N_j-1}{r}\delta_1^r \delta_2 +
4N_j^2 \varepsilon_1^{N_j-1}\delta_3
+ 1
\right) - 1  \\
\leq & \exp\left(\sum_{{j\in S}}\mathbb{E} \left(
2 \sum_{r=2}^{N_j}  \binom{N_j}{r} \delta_2^r + 4N_j^2 \sum_{r=1}^{N_j-1} \binom{N_j-1}{r}\delta_1^r \delta_2 +
4N_j^2 \varepsilon_1^{N_j-1}\delta_3
\right) \right) -1
\end{align*}
{where the equality follows from $\sum_{\emptyset\ne J\subseteq S} \prod_{j \in J} x_j= \prod_{j\in S}(x_j+1) -1$ again, and the final step uses $\prod_{j\in S} (1+x_j) \leq \exp(\sum_{j\in S} x_j)$ for $x_j \geq 0$.}

{Finally, under the assumptions $\bar{p}/p_\lambda \leq \varepsilon_1$ and $\max_{j\in S} np_j(1 \vee \delta_1 \vee \delta_2) \leq \varepsilon_2$, the same calculations used in \cite{robins2009semiparametric} imply that}
\begin{align*}
\sum_{{j\in S}}\mathbb{E}\left[
2 \sum_{r=2}^{N_j}  \binom{N_j}{r} \delta_2^r + 4N_j^2  \sum_{r=1}^{N_j-1} \binom{N_j-1}{r}\delta_1^r \delta_2 +
4N_j^2 \varepsilon_1^{N_j-1}\delta_3
\right] \\
\leq
C n \sum_{j \in S} p_j \left\{n \max_{j \in S} p_j (\delta_2^2+\delta_1 \delta_2) + \delta_3\right\}
\end{align*}
where $C$ depends only on $\varepsilon_1$ and $\varepsilon_2$. Substituting the bound then gives
\[
\chi^2 (
\int Q_\lambda^n d \pi (\lambda),
\int P_\lambda^n d \pi (\lambda)
)  \leq
\exp \left(
C n \sum_{j \in S} p_j \left\{n \max_{j \in S} p_j (\delta_2^2+\delta_1 \delta_2) + \delta_3\right\}
\right)
-1.
\]
\end{proof}

\subsection{Proof for Theorem~\ref{lemma::lower-bound}} \label{app::lower-bound}

\begin{proof}
For each $i$, write $\overline P_i=\int (P^{(i)})_\lambda^n\,d\pi(\lambda)$. Given an estimator $\widehat{\mathbb S}$, define the induced test $\psi$ to be an index minimizing $\operatorname{dist}(\widehat{\mathbb S},\mathbb S_{(i)})$. Because the targets are separated by at least $2r$, the event $\{\psi\ne i\}$ is contained in $\{\operatorname{dist}(\widehat{\mathbb S},\mathbb S_{(i)})\ge r\}$. Since $(P^{{(i)}})_\lambda\in\mathcal P$ for all $ i = 0, \cdots, M$ and $\lambda \in \Lambda$,
    \begin{align*}
    \inf_{\hat{\mathbb{S}}}
    \sup_{P\in\mathcal P}
    \mathbb E_{P^{ n}}
    \!\left[
    \operatorname{dist}(\hat{\mathbb{S}},\mathbb{S}(P))
    \right]
    \;\ge\;&
    \inf_{\hat{\mathbb{S}}}
    \sup_{\substack{0\le i\le M\\ \lambda\in \Lambda}}
    \mathbb E_{(P^{{(i)}})_\lambda^{ n}}
    \!\left[
    \operatorname{dist}(\hat{\mathbb{S}},\mathbb{S}_{{(i)}})
    \right]
    \\[0.3em]
    \;\ge\;&
    \inf_{\hat{\mathbb{S}}}
    \sup_{0\le i\le M}
    \int
    \mathbb E_{(P^{{(i)}})_\lambda^{ n}}
    \!\left[
    \operatorname{dist}(\hat{\mathbb{S}},\mathbb{S}_{{(i)}})
    \right]
    \,d\pi(\lambda)\\
    \geq &  r \cdot \inf_{\hat{\mathbb{S}}} \sup_{\substack{0 \leq i \leq M}} \int
    {(P^{{(i)}})_\lambda^n} \left( \operatorname{dist}(\hat{\mathbb{S}}, \mathbb{S}_{{(i)}}) \geq r \right) d\pi(\lambda)  \\
    \geq & r \cdot  \inf_{\psi}\sup_{0\le i\le M}\overline P_i(\psi\ne i)\\
    \geq & r \cdot  \frac{1}{2} (1-c-\frac{1}{M}),
    \end{align*}
where the third line uses the inequality $u\ge r\mathbbm{1}\{u\ge r\}$ for $u\ge0$, the fourth line uses the induced test above, and the final line follows from the $\chi^2$ testing bound in Propositions~2.2 and~2.4 and Theorem~2.6 of \cite{Tsybakov2009intro}.
\end{proof}

\section{Proofs for Upper bound}

\subsection{Proof for Proposition~\ref{lemma::basis}} \label{app::basis}
\begin{proof}
For $x \in \mathcal{X}$, define $s_k^x = \int_{ [0,1]^d} b_k(v)K_k(v) g(v)dF^{*,x}(v)$. We first show $\|s_k^x\| \lesssim 1$. By construction,
    \begin{align*}
    \int_{[0,1]^d} ( g(v) -  b_k(v)^T \Omega_x^{-1}  s_k^x)^2 K_{k}(v) dF^{*,x}(v)
    =
    \int_{[0,1]^d}
    g(v) ^2 K_{k}(v) dF^{*,x}(v)   -(s_k^x)^T\Omega_x^{-1} s_k^x  \geq 0.
\end{align*}
Hence,
\[
(s_k^x)^T \Omega^{-1}_x s_k^x \lesssim 1
\]
by the boundedness of $K_k$, $g$ and $dF^{*,x}$. Since $\lambda_{\max}(\Omega_x) \lesssim 1$ and
 \[
 (s_k^x)^T \Omega^{-1}_x s_k^x \geq
 {\|s_k^x\|^2 }/{\lambda_{\max}(\Omega_x)}
   ,
 \]
it follows that $\sup_{x \in \mathcal{X}} \|s_k^x\| \lesssim 1$. Therefore,
 \begin{align*}
    \sup_{x \in \mathcal{X}}  \|\hat{\Pi}_x g\|_{K \cdot F^{*,x}}^2
     & =
     \sup_{x \in \mathcal{X}}\int (b_k(v)^T \hat{\Omega}_x^{-1}s_k^x)^2K_{k}(v) dF^{*,x}(v) \\
     & = \sup_{x \in \mathcal{X}}(s_k^x)^T \hat{\Omega}_x^{-1} \Omega_x
     \hat{\Omega}_x^{-1}
     s_k^x
     \\
     & \leq
     \sup_{x \in \mathcal{X}}\|\hat{\Omega}_x^{-1}\|
     \|{\Omega}_x\|
     \|\hat{\Omega}_x^{-1}\| \|s_k^x\|^2 \lesssim 1.
 \end{align*}
where the boundedness of $\|\hat{\Omega}_x^{-1}\|$ follows from Lemma~\ref{lemma::eigen-omegahat}.

Moreover, condition (\ref{eq::b-support}) implies that for every $j = 1, \cdots, k$,
  \[
  |[s_k^x]_j| = |\int_{[0,1]^d}
  [b_k]_j(v) K_k(v) g(v) dF^{*,x}(v) |
  \lesssim
  k^{-1/2}
  \]
so that $\|s_k^x\|_{\infty} \lesssim k^{-1/2}$. Consequently,
 \[
 \sup_{x \in \mathcal{X}} \sup_{v \in [0,1]^d} | \Pi_xg (v) |
 \leq
 \sup_{x \in \mathcal{X}}
 \sup_{v \in [0,1]^d} \|b_k(v) \|_{1} \|\Omega_x^{-1}\|_{\infty} \|s_k^x\|_{\infty} \lesssim 1,
 \]
where we used condition (\ref{eq::b-sup}) and (\ref{eq::b-row-sum}). The same argument gives
 \[
 \sup_{x \in \mathcal{X}} \sup_{v \in [0,1]^d} | \hat{\Pi}_x g (v) |
 \leq
 \sup_{x \in \mathcal{X}}
 \sup_{v \in [0,1]^d} \|b_k(v) \|_{1} \|\hat{\Omega}_x^{-1}\|_{\infty} \|s_k^x\|_{\infty} \lesssim 1.
 \]
This completes the proof.
\end{proof}

\subsection{Proof for Proposition~\ref{prop::decomposition}} \label{app::decomposition}
\begin{proof}
We begin with the $L_2(\mathcal{X})$ risk. By the uniform  boundedness of $\|Q_x^{-1}\|$ and $\|\hat{Q}_x^{-1} - Q_x^{-1}\|$, (\ref{eq::decompose-tau}) implies
\[
|\hat{\tau}(x) - \tau_h(x)|^2 \lesssim \|\hat{S}_x - S_x\|^2.
\]
Therefore, by Tonelli's theorem,
\begin{align*}
    \mathbb{E}
    \|\hat{\tau}-\tau_h\|^2_{L_2(\mathcal{X})}
    =
    \int_{\mathcal{X}}
    \mathbb{E} \left(
     (\hat{\tau}(x)-\tau_h(x))^2
    \right)
    dx
    \lesssim &
    \int_{\mathcal{X}}
    \mathbb{E}\|\hat{S}_x - S_x\|^2 dx
    \leq
    \sup_{x \in \mathcal{X}}
    \mathbb{E}\|\hat{S}_x - S_x\|^2 .
\end{align*}
Since the dimension $q$ is fixed,
\[
\mathbb{E}\|\hat{S}_x - S_x\|^2  \lesssim \max_{1 \leq j \leq q} \mathbb{E}
    \left(
    [\hat{S}_x - S_x]_j^2
    \right).
\]
Applying the conditional bias-variance decomposition,
\[
\mathbb{E}
    \left(
    [\hat{S}_x - S_x]_j^2
    \right)
    = \mathbb{E}
    \left \{
    \mathbb{E}
    \left(
    [\hat{S}_x - S_x]_j \mid D^n
    \right)^2 +
    \operatorname{Var}([\hat{S}_x]_j \mid D^n )
    \right\}
\]
Hence,
\begin{align*}
     \mathbb{E}
    \|\hat{\tau}-\tau_h\|^2_{L_2(\mathcal{X})}
    \lesssim &
    \sup_{x \in \mathcal{X}} \max_{1 \leq j \leq q}
    \mathbb{E}
    \left \{
    \mathbb{E}
    \left(
    [\hat{S}_x - S_x]_j \mid D^n
    \right)^2 +
    \operatorname{Var}([\hat{S}_x]_j \mid D^n )
    \right\}.
\end{align*}
We next consider the $L_\infty(\mathcal{X})$ risk. Again by (\ref{eq::decompose-tau}),
\[
\mathbb{E}\|\hat{\tau}-\tau_h\|_{L_\infty(\mathcal{X})}^2
\lesssim
\mathbb{E}
\left( \sup_{x \in \mathcal{X}}
\|\hat{S}_x - S_x\|^2
\right).
\]
Decompose
\[
\hat{S}_x - S_x = \hat{S}_x -\mathbb{E}(\hat{S}_x  \mid D^n)  + \mathbb{E}(\hat{S}_x  \mid D^n)  - S_x.
\]
Using $\sup_x\|a_x+b_x\|^2 \leq 2\sup_x\|a_x\|^2 + 2\sup_x\|b_x\|^2$, we obtain
\begin{align*}
    \mathbb{E} \|\hat{\tau}-&\tau_h\|_{L_\infty(\mathcal{X})}^2
    \lesssim
    \mathbb{E} \left(
    \sup_{x \in \mathcal{X}}\|\hat{S}_x -
    \mathbb{E}(\hat{S}_x  \mid D^n) \|^2
    \right)  +
    \mathbb{E}
    \left(
    \sup_{x \in \mathcal{X}}
    \|\mathbb{E}(\hat{S}_x\mid D^n) -S_x\|^2
    \right).\\
    &
    \lesssim
    \mathbb{E}  \left\{
    \mathbb{E} \left( \sup_{x \in \mathcal{X}} \|\hat{S}_x -
    \mathbb{E}(\hat{S}_x  \mid D^n) \|^2
    \mid D^n  \right)
    \right\} +
    \mathbb{E}  \left\{
     \sup_{x \in \mathcal{X}}
     \max_{1 \leq j \leq q}
     \left|
     \mathbb{E} \left( [\hat{S}_x-S_x]_j \mid D^n \right)
     \right|^2
      \right\}
\end{align*}
where the second line follows from the tower property and the fact that $q$ is fixed. To bound the first term, let $x_1,\dots,x_M$ be a maximum packing of $\mathcal{X}$. Then every $x\in \mathcal{X}$ satisfies
\[
\min_{1\le \ell \le M}\|x-x_\ell\|\lesssim  M^{-1/d}.
\]
A standard discretization argument therefore yields
\begin{align*}
    \mathbb{E} \left(
    \sup_{x \in \mathcal{X}}
    \|\hat{S}_x -
    \mathbb{E}(\hat{S}_x  \mid D^n) \|^2  \mid D^n  \right)
    \lesssim &
    \mathbb{E} \left(
    \max_{1\le \ell \le M}
    \|\hat S_{x_\ell}-\mathbb{E}(\hat S_{x_\ell}\mid D^n)\|^2
    \mid D^n
    \right) \\
    & +
    \mathbb{E}\left(
    \sup_{x, x^\prime \in \mathcal{X}: \|x-x'\| \lesssim  M^{-1/d}}
    \|(\hat S_x-\hat S_{x'})-
    \mathbb{E}(\hat S_x-\hat S_{x'}\mid D^n)\|^2
    \mid D^n
    \right)
\end{align*}
Since $q$ is fixed, the first term is further bounded by
\[
    \mathbb{E} \left(
    \max_{1\le \ell \le M}
    \max_{1 \le j \le q}
    \left|[\hat S_{x_\ell}]_j-\mathbb{E}([\hat S_{x_\ell}]_j \mid D^n) \right|^2
    \mid D^n
    \right).
\]
As for the second term, Jensen's inequality implies
\[
\|\mathbb{E}(\hat{S}_x - \hat{S}_{x^\prime} \mid D^n)\|^2  \leq \mathbb{E} \left( \|\hat{S}_x - \hat{S}_{x^\prime} \|^2 \mid D^n
\right).
\]
Hence,
\begin{align*}
    \mathbb{E}\left(
    \sup_{x, x^\prime \in \mathcal{X}: \|x-x'\| \lesssim  M^{-1/d}}
    \|(\hat S_x-\hat S_{x'})-
    \mathbb{E}(\hat S_x-\hat S_{x'}\mid D^n)\|^2
    \mid D^n
    \right)  \\
    \lesssim
    \mathbb{E}\left(
    \sup_{x, x' \in \mathcal{X}: \|x-x'\|\lesssim M^{-1/d}}
    \|\hat S_x-\hat S_{x'}\|^2
    \mid D^n
    \right).
\end{align*}
Combining the preceding bounds, we obtain
    \begin{align*}
    \mathbb{E}&\|\hat{\tau}-\tau_h\|_{L_\infty(\mathcal{X})}^2  \lesssim
    \mathbb{E} \left\{
    \mathbb{E} \left(
    \max_{1\le \ell \le M}
    \max_{1 \le j \le q}
    \left|[\hat S_{x_\ell }]_j-\mathbb{E}([\hat S_{x_\ell} ]_j\mid D^n) \right|^2
    \mid D^n
    \right) \right. \\
    &
    \left.
    \quad +
     \mathbb{E}\left(
    \sup_{x, x' \in \mathcal{X}: \|x-x'\|\lesssim M^{-1/d}}
    \|\hat S_x-\hat S_{x'}\|^2
    \mid D^n
    \right)
    +
    \sup_{x \in \mathcal{X}}
     \max_{1 \leq j \leq q}
     \left|
     \mathbb{E} \left( [\hat{S}_x - S_x]_j \mid D^n \right)
     \right|^2
    \right\}.
    \end{align*}

\end{proof}

\subsection{Proof for Proposition~\ref{prop::l_inf_1}} \label{app::l_inf_1}
\begin{proof}
For any fixed $x \in \mathcal{X}$, the quantity $[\hat{S}_{x}]_j - \mathbb{E}( [\hat{S}_{x}]_j \mid D^n)$ is a centered U-statistic with non-degenerate kernel
\begin{align*}
    [\xi^x]_j(Z_1, Z_2) =& [\rho_h^x(X_1)]_j K_{hk}^x(X_1)
    \left\{
    \hat{\psi}_{y_1}(Z_1) - \hat{\psi}_{a_1}(Z_1) \rho_h^x(X_1)^T \theta_x
    \right\} \\
    & \quad +
    [\rho_h^x(X_1)]_j  K_{hk}^x(X_1)
    \left\{
    \hat{\psi}_{y_2}^x(Z_1, Z_2) - \hat{\psi}_{a_2}^x(Z_1, Z_2) \rho_h^x(X_2)^T \theta_x
    \right\}K_{hk}^x(X_2).
\end{align*}
Conditional on the training sample $D^n$, which is suppressed from the notation throughout the proof, define the first-order projections in the Hoeffding decomposition by
\begin{equation} \label{eq::deg-kernel1}
    [\tilde{\xi}_1^x]_j(Z_1) =  \mathbb{E}_{Z_2}\left( [\xi^x]_j (Z_1, Z_2) \mid Z_1\right)- \mathbb{E}\left( [\xi^x]_j (Z_1, Z_2) \right),
\end{equation}
\begin{equation} \label{eq::deg-kernel2}
    [\tilde{\xi}_2^x]_j(Z_2) =  \mathbb{E}_{Z_1}\left( [\xi^x]_j (Z_1, Z_2) \mid Z_2 \right)- \mathbb{E}\left( [\xi^x]_j (Z_1, Z_2) \right),
\end{equation}
and the completely degenerate second-order kernel
\begin{equation} \label{eq::deg-kernel3}
    [\tilde{\xi}^x]_j(Z_1, Z_2) = [\xi^x]_j(Z_1, Z_2) - \mathbb{E}_{Z_2}\left( [\xi^x]_j (Z_1, Z_2) \mid Z_1\right) -
    \mathbb{E}_{Z_1}\left( [\xi^x]_j (Z_1, Z_2) \mid Z_2 \right) +
    \mathbb{E} \left( [\xi^x]_j (Z_1, Z_2) \right).
\end{equation}
Hence
\[
[{\xi}^x]_j(Z_1, Z_2) - \mathbb{E} \left( [\xi^x]_j (Z_1, Z_2) \right) =
[\tilde{\xi}^x]_j(Z_1, Z_2) +
[\tilde{\xi}_1^x]_j (Z_1) +
[\tilde{\xi}_2^x]_j (Z_2).
\]
Therefore, bounding
\[
 \mathbb{E}[
 \max_{1 \leq \ell \leq M} \max_{1 \leq j \leq q}
 |
 [\hat{S}_{x_\ell}]_j -[ \mathbb{E}( \hat{S}_{x_\ell} \mid D^n) ]_j
 |^2
 \mid D^n]
\]
reduces to controlling
\begin{equation} \label{eq::target-kernel}
    \mathbb{E} \left(
    \max_{1 \leq \ell \leq M}\max_{1 \leq j \leq q} | \mathbb{P}_n([\tilde{\xi}_1^{x_\ell}]_j) |^2
    +
    \max_{1 \leq \ell \leq M}\max_{1 \leq j \leq q} | \mathbb{P}_n([\tilde{\xi}_2^{x_\ell}]_j) |^2 +
    \max_{1 \leq \ell \leq M}\max_{1 \leq j \leq q} | \mathbb{U}_n([\tilde{\xi}^{x_\ell}]_j) |^2 \mid D^n
    \right)
\end{equation}
We next invoke the following concentration inequality for empirical processes and degenerate U-statistics.
\begin{lemma} (Proposition 2.3 in \cite{arcones1993limit}, Theorem 3.4.8 from \cite{gine2021mathematical}) \label{lemma::concentration}
Let $f$ be a centered first-order kernel and $h$ a completely degenerate second-order kernel satisfying
    \[
        \|f\|_{ L_\infty(\mathcal{Z}) } \leq b,
        \quad
        \sigma^2 = \mathbb{E} \left( f(Z)^2 \right), \quad
        A = \|h(z_1, z_2)\|_{L_\infty(\mathcal{Z}\times\mathcal{Z})},
        \quad
        C^2 = \frac{n(n-1)}{2} \mathbb{E}\left(h(Z_1, Z_2)^2\right),
    \]
    \[
        B^2 = (n-1) \max \left\{
        \|\mathbb{E}_{Z_1}\left(h(Z_1, z_2)^2\right)\|_{L_\infty(\mathcal{Z})},
        \|\mathbb{E}_{Z_2}\left(h(z_1, Z_2)^2\right)\|_{L_\infty(\mathcal{Z})}
        \right\},
        \]
    \[
        D = \frac{n}{2} \sup
        \left\{ \mathbb{E}\left(
        h(Z_1, Z_2) g_1(Z_1) g_2(Z_2) \right):
        \mathbb{E}\left(g_1(Z)^2 \right),
        \mathbb{E}\left(g_2(Z)^2 \right) \leq 1
        \right\}.
    \]
Then for every $t >0$,
    \[
    \Pr(|\mathbb{P}_n(f)| > t) \lesssim
    e^{- \frac{nt^2}{2\sigma^2 + (2/3)bt}}.
    \]
    \[
    \Pr( |\mathbb{U}_n(h)| \geq t) \lesssim e^{- c \min \left( \frac{n^4t^2}{C^2}, \frac{n^2t}{D}, \frac{n^{4/3}t^{2/3}}{B^{2/3}}, \frac{nt^{1/2}}{A^{1/2}}\right)}
    \]
for some universal constant $c$.
\end{lemma}
To apply the result, it remains to derive the quantities appearing in the concentration bounds for the kernels $\tilde{\xi}_1^x$,  $\tilde{\xi}_2^x$ and $\tilde{\xi}^x$. The required quantities are summarized in the following proposition.

\begin{proposition} \label{prop::kernel-bound}
For every $x \in \mathcal{X}$,
    \[
    \|[\tilde{\xi}_1^x]_j \|_{L_\infty(\mathcal{Z})}, \|[\tilde{\xi}_2^x]_j\|_{L_\infty(\mathcal{Z})}
    \lesssim \frac{1}{h^d}, \
    \sigma^2([\tilde{\xi}_1^x]_j),
    \sigma^2([\tilde{\xi}_2^x]_j)
    \lesssim \frac{1}{h^d}, \
    \| [\tilde{\xi}^x]_j \|_{L_\infty(\mathcal{Z})} \lesssim \frac{k}{h^{2d}},  \
    \sigma^2([\tilde{\xi}^x]_j) \lesssim \frac{k}{h^{2d}},
    \]
    \[
    \left \|
    \mathbb{E}_{Z_1}\left( [\tilde{\xi}^x]_j(Z_1, z_2)^2\right)\right  \|_{L_\infty(\mathcal{Z})}
    ,
    \left \|\mathbb{E}_{Z_2}\left( [\tilde{\xi}^x]_j(z_1, Z_2)^2\right)  \right  \|_{L_\infty(\mathcal{Z})}
    \lesssim \frac{k}{h^{3d}},
    \]
and
    \[
    \sup\left\{ \mathbb{E}\left(
        [\tilde{\xi}^x]_j(Z_1, Z_2) g_1(Z_1) g_2(Z_2) \right):
        \mathbb{E}\left(g_1(Z)^2\right),
        \mathbb{E}\left(g_2(Z)^2\right) \leq 1
        \right\} \lesssim \frac{1}{h^d}
    \]
\end{proposition}
\begin{proof}
The proof is provided in Appendix~\ref{app::kernel-bound}
\end{proof}
Applying Lemma~\ref{lemma::concentration} together with Lemma~\ref{lemma::maximum}, we first bound the contribution of the completely degenerate second-order kernel $[\tilde{\xi}^x]_j$. Substituting the bounds from Proposition~\ref{prop::kernel-bound} into the tail bound for degenerate U-statistics yields
\begin{align}
\nonumber
    \mathbb{E} \left(
    \max_{1 \leq \ell \leq M}
    \right. & \left.\max_{1 \leq j \leq q} |\mathbb{U}_n([\tilde{\xi}^{x_\ell}]_j)|^2 \mid D^n
    \right) \\
\nonumber
    \lesssim &
    \frac{1}{n^4} [
    \frac{n^2k }{h^{2d}}\log(Mq)
    +
    \frac{n^2 }{h^{2d}} (\log(Mq))^2 +
    {\frac{nk}{h^{3d}}} (\log(Mq))^{3} +
    \frac{k^2}{h^{4d}} (\log(Mq))^4
    ] \\
\nonumber
    \lesssim &
    \frac{1}{n^4} [
    \frac{n^2k }{h^{2d}}\log n
    +
    \frac{n^2}{h^{2d}} (\log n)^2 +
    {\frac{nk}{h^{3d}}} (\log n)^{3} +
    \frac{k^2}{h^{4d}} (\log n)^4
    ] \\
\label{eq::bound-second}
    \lesssim &
    \frac{{k} \log n}{n^2 h^{2d}} ,
\end{align}
under the assumptions that $Mq$ is polynomial in $n$ of fixed order, $nh^d \gtrsim (\log n)^2$, $k\gtrsim\log n$, and $k (\log n)^3 \lesssim (nh^d)^2$.

Applying the same argument to the empirical-process terms associated with  $[\tilde{\xi}_1^x]_j$ and $[\tilde{\xi}_2^x]_j$ gives
\begin{equation}
    \mathbb{E}\left( \max_{1 \leq \ell \leq M} \max_{1 \leq j \leq q} | \mathbb{P}_n([\tilde{\xi}_i^{x_\ell}]_j(Z)) |^2 \mid D^n  \right) \lesssim
    \frac{\log n}{nh^d} + \frac{(\log n)^2}{n^2 h^{2d}} \lesssim \frac{\log n}{nh^d}
\end{equation}
under the assumption. Combining the preceding bounds then yields
\begin{align*}
 \mathbb{E}\left(
    \max_{1 \leq \ell \leq M} \max_{1 \leq j \leq q}
    \left|
    [\hat{S}_{x_\ell}]_j - \mathbb{E}([\hat{S}_{x_\ell}]_j \mid D^n)
    \right|^2 \mid D^n
    \right)
    \lesssim
    \frac{k \log n}{n^2 h^{2d}} + \frac{\log n}{nh^d}.
\end{align*}
\end{proof}

\subsubsection{Proof for Proposition~\ref{prop::kernel-bound}} \label{app::kernel-bound}
\begin{proof}
For a generic function $g$, define its localized version at $x$ by $\tilde{g}^{*,x}(v) = g(x+h(v-1/2))$. Let $[\xi_1^x]_j(Z_1) = \mathbb{E}_{Z_2}\left( [\xi^x]_j (Z_1, Z_2) \mid Z_1\right)$ and $[\xi_2^x]_j(Z_2) = \mathbb{E}_{Z_1}\left( [\xi^x]_j (Z_1, Z_2) \mid Z_2\right)$. A direct calculation yields
\begin{align*}
    [\xi_1^x]_j(Z_1)  = & [\rho_h]_j^x(X_1)K_{hk}^x(X_1)
    \left( A_1 - \hat{\pi}(X_1)\right)
    \left(Y_1 - \hat{\mu}_0(X_1)-A_1\tilde{\tau}_h^x(X_1) \right)  \\
    -
    [\rho_h]_j^x&(X_1)  K_{hk}^x(X_1) (A_1 - \hat{\pi}(X_1)) \hat{\Pi}_x(\mu_0^{*,x} - \hat{\mu}_0^{*,x} + {\pi}^{*,x} \tau^{*,x}-\pi^{*,x}(\tilde{\tau}_h^x)^{*,x})(\frac{X_1-x}{h}+\frac{1}{2}),
\end{align*}
and
\begin{align*}
    [\xi_2^x]_j&(Z_2)
    = \int_{[0,1]^d}
    ([\rho]_j\pi^{*,x}-[\rho]_j\hat{\pi}^{*,x})(v) \cdot(\mu_0^{*,x}-\hat{\mu}_0^{*,x} + \pi^{*,x}\tau^{*,x}-\pi^{*,x}(\tilde{\tau}_h^x)^{*,x})(v){K_k(v)\,}dF^{*,x}(v) \\
    & \quad - \hat{\Pi}_x([\rho]_j\pi^{*,x}-[\rho]_j\hat{\pi}^{*,x})(\frac{X_2-x}{h}+\frac{1}{2}) \left(Y_2 - \hat{\mu}_0(X_2)-A_2\tilde{\tau}_h^x(X_2) \right)K_{hk}^x(X_2).
\end{align*}
where $\tilde{\tau}_h^x$ is defined in (\ref{eq::tilde-tau}). Since  $\hat{\mu}_0$, $\hat{\pi}$, and $y$ are uniformly bounded, Proposition~\ref{lemma::basis} implies
\[
\|[\xi_1^x]_j\|_{ L_\infty(\mathcal{Z}) }, \|[\xi_2^x]_j\|_{ L_\infty(\mathcal{Z}) } \lesssim \frac{1}{h^d},
\]
and by condition (\ref{eq::b-sup}),
\[
\|[\xi^x]_j\|_{ L_\infty(\mathcal{Z}) } \lesssim \frac{k}{h^{2d}}.
\]
Therefore,
\[
\|[\tilde{\xi}_i^x]_j\|_{L_\infty(\mathcal Z)}
\leq
\|[\xi_i^x]_j\|_{L_\infty(\mathcal Z)}
+|\mathbb E([\xi_i^x]_j)|
\lesssim
\frac{1}{h^d}
+\left(\frac{h}{k^{1/d}}\right)^{2s}
+h^\gamma
\lesssim
\frac{1}{h^d},
\quad i=1,2,
\]
where the expectation term is bounded by Proposition~\ref{prop::bias-refined}. Similarly,
\[
    \|[\tilde{\xi}^x]_j\|_{L_\infty(\mathcal{Z})}
    \lesssim
    \|[\xi^x]_j\|_{L_\infty(\mathcal{Z})} +
    \|[\xi^x_1]_j\|_{L_\infty(\mathcal{Z})}  +
    \|[\xi^x_2]_j\|_{L_\infty(\mathcal{Z})}  +
    |\mathbb{E}([\xi]_j )| \lesssim \frac{k}{h^{2d}}.
\]
Next, Proposition~7 of \cite{kennedy2024minimax} gives
\begin{equation} \label{eq::sigma-1}
\sigma^2([\tilde{\xi}_1^x]_j) \lesssim
\mathbb{E}(([{\xi}_1^x]_j)^2) \lesssim \frac{1}{h^d}.
\end{equation}
As for $[\tilde{\xi}_2^x]_j$,
\begin{align}
    \nonumber
    \sigma^2([\tilde{\xi}_2^x]_j)
    =&\operatorname{Var}_{Z_2}\left(
    \hat{\Pi}_x([\rho]_j\pi^{*,x}-[\rho]_j\hat{\pi}^{*,x})(\frac{X_2-x}{h}+\frac{1}{2}) \left(Y_2 - \hat{\mu}_0(X_2)-A_2\tilde{\tau}_h^x(X_2) \right)K_{hk}^x(X_2)
    \right) \\
    \nonumber
    \lesssim &
    \mathbb{E}_{Z_2}\left(
    \hat{\Pi}_x([\rho]_j\pi^{*,x}-[\rho]_j\hat{\pi}^{*,x})(\frac{X_2-x}{h}+\frac{1}{2})^2 K_{hk}^x(X_2)^2
    \right) \\
    \nonumber
    \leq &
    \frac{1}{h^{d}}
    \int_{[0,1]^d}  \hat{\Pi}_x([\rho]_j\pi^{*,x}-[\rho]_j\hat{\pi}^{*,x})(v)^2K_{k}(v)dF^{*,x}(v) \\
    \nonumber
    = &
    \frac{1}{h^d}\|\hat{\Pi}_x([\rho]_j\pi^{*,x}-[\rho]_j\hat{\pi}^{*,x})\|^2_{K\cdot F^{*,x}}\\
    \label{eq::sigma-2}
    \lesssim & \frac{1}{h^{d}}
\end{align}
where the last inequality follows from Proposition~\ref{lemma::basis} applied to the bounded function $[\rho]_j\pi^{*,x}-[\rho]_j\hat{\pi}^{*,x}$ . Furthermore,
\begin{align*}
    \left \| \mathbb{E}_{Z_1} \left( [\tilde{\xi}^x]_j(Z_1,z_2)^2
    \right)
    \right \|_{L_\infty(\mathcal{Z})}
    =&
    \left \|
    \mathbb{E}_{Z_1}
    \left(
    \left(
    [{\xi}^x]_j(Z_1,z_2) - [\xi_2^x]_j(z_2)
    - [\xi_1^x]_j(Z_1)  + \mathbb{E}( [\xi^x]_j)
    \right)^2
    \right)\
    \right \|_{L_\infty(\mathcal{Z})} \\
    \lesssim &
     \left \| \mathbb{E}_{Z_1}
     \left(
     \left( [{\xi}^x]_j(Z_1, z_2) -
    [\xi_2^x]_j(z_2) \right)^2
    \right)
    \right \|_{L_\infty(\mathcal{Z})}
    + \sigma^2([\xi_1^x]_j).
\end{align*}
The first term on the right hand side can be represented as
\begin{align*}
    \left \| \mathbb{E}_{Z_1} \left( \left( [{\xi}^x]_j(Z_1, z_2) -
    [\xi_2^x]_j(z_2) \right)^2 \right)
    \right \|_{L_\infty(\mathcal{Z})}
    = &
    \|
    \operatorname{Var}_{Z_1}
    \left(
    [{\xi}^x]_j(Z_1, z_2)
    \right) \|_{L_\infty(\mathcal{Z})}
    \\
    \lesssim &
    \left \|
    \mathbb{E}_{Z_1} \left(
    \left( [
    \rho_h^x(X_1)]_j K_{hk}^x(X_1) \right)^2
    \right)
    \right \|_{L_\infty(\mathcal{Z})}
     \\
     & \quad +
    \left \|
    \mathbb{E}_{Z_1} \left(
    (
    b_{hk}^x(X_1)^T \hat{\Omega}_x^{-1} b_{hk}^x(x_2)
    )^2
     K_{hk}^x(X_1)^2
    \right)
    K_{hk}^x(x_2)^2
    \right \|_{L_\infty(\mathcal{Z})} \\
    \lesssim
    &
    \frac{1}{h^d} + \frac{1}{h^d}\sup_{\|x_2-x\|\leq h/2}
    \left|
    b_{hk}^x(x_2)^T\hat{\Omega}_x^{-1}\Omega_x
    \hat{\Omega}_x^{-1} b_{hk}^x(x_2) K_{hk}^x(x_2)^2
    \right|
    \\
    \lesssim &
    \frac{k}{h^{3d}},
\end{align*}
where the first inequality uses the boundedness of $y^2$, $\hat{\mu}_0^2$ and $\hat{\pi}^2$, while the last inequality follows from condition~(\ref{eq::b-sup}), Lemma~\ref{lemma::eigen-omegahat}, and submultiplicativity of the operator norm. An identical argument yields
\begin{align*}
     \left \|  \mathbb{E}_{Z_2} \left( [\tilde{\xi}^x]_j(z_1,Z_2)^2
    \right)
    \right \|_{L_\infty(\mathcal{Z})}
    \lesssim &
    \frac{k}{h^{3d}}
    + \sigma^2([\xi_2^x]_j).
\end{align*}
Combining the bounds with (\ref{eq::sigma-1}) and (\ref{eq::sigma-2}) then yields
\[
    \left \|
    \mathbb{E}_{Z_1}\left( [\tilde{\xi}^x]_j(Z_1, z_2)^2\right)\right  \|_{L_\infty(\mathcal{Z})}
    ,
    \left \|\mathbb{E}_{Z_2}\left( [\tilde{\xi}^x]_j(z_1, Z_2)^2\right)  \right  \|_{L_\infty(\mathcal{Z})}
    \lesssim \frac{k}{h^{3d}}.
\]

Moreover, Proposition~7 of \cite{kennedy2024minimax} shows that $\mathbb{E}([\xi^x]_j(Z_1,Z_2)^2)\lesssim \frac{k}{h^{2d}} $. Therefore,
\[
\sigma^2([\tilde{\xi}^x]_j)
\lesssim
\mathbb{E}
\left( [\xi^x]_j(Z_1, Z_2)^2 \right) + \sigma^2([\xi_1^x]_j) + \sigma^2([\xi_2^x]_j) + |\mathbb{E}([\xi^x]_j)|^2 \lesssim  \frac{k}{h^{2d}}.
\]

Finally, define
\[
u_{hk}^x(Z) = [\rho_h^x(X)]_j K_{hk}^x(X) (A-\hat{\pi}(X)) b_{hk}^x(X),
\]
and
\begin{equation}\label{eq::kernel-outcome-factor}
v_{hk}^x(Z) =  K_{hk}^x(X)(Y-\hat{\mu}_0(X)-A\rho_h^x(X)^T \theta_x) b_{hk}^x(X).
\end{equation}

Suppressing the dependence on $x$, $h$, and $k$, let
\[
\mu_u=\mathbb E\{u_{hk}^x(Z)\},
\qquad
\mu_v=\mathbb E\{v_{hk}^x(Z)\},
\]
and denote their $\ell$th coordinates by $[\mu_u]_\ell$ and $[\mu_v]_\ell$, respectively. By the boundedness of the density and condition~\eqref{eq::b-support},
\begin{equation}\label{eq::uj}
\mathbb E\left\{[u_{hk}^x]_\ell(Z)^2\right\}
\lesssim
h^{-d}
\int_{[0,1]^d}[b_k]_\ell(v)^2\,dF^{*,x}(v)
\lesssim h^{-d}.
\end{equation}
Moreover,
\[
[\mu_u]_\ell^2
\lesssim
\left(
\int_{[0,1]^d}|[b_k]_\ell(v)|\,dv
\right)^2
\lesssim k^{-1},
\]
which implies
\begin{equation}\label{eq::mu-u}
\|\mu_u\|^2
=
\sum_{\ell=1}^k[\mu_u]_\ell^2
\lesssim 1.
\end{equation}
The same argument gives
\[
\mathbb E\left\{[v_{hk}^x]_\ell(Z)^2\right\}
\lesssim h^{-d},
\qquad
[\mu_v]_\ell^2\lesssim k^{-1},
\qquad
\|\mu_v\|^2\lesssim1.
\]
The degenerate kernel admits the representation
\[
[\tilde\xi^x]_j(Z_1,Z_2)
=
-\left\{u_{hk}^x(Z_1)-\mu_u\right\}^T
\hat\Omega_x^{-1}
\left\{v_{hk}^x(Z_2)-\mu_v\right\}.
\]
For any measurable functions $g_1$ and $g_2$ satisfying
\[
\mathbb E\{g_1(Z_1)^2\}\leq1,
\qquad
\mathbb E\{g_2(Z_2)^2\}\leq1,
\]
independence of $Z_1$ and $Z_2$ gives
\begin{align*}
&\left|
\mathbb E\left\{
[\tilde\xi^x]_j(Z_1,Z_2)
g_1(Z_1)g_2(Z_2)
\right\}
\right|\\
&\quad=
\left|
\mathbb E_{Z_1}\left[
g_1(Z_1)\{u_{hk}^x(Z_1)-\mu_u\}
\right]^T
\hat\Omega_x^{-1}
\mathbb E_{Z_2}\left[
g_2(Z_2)\{v_{hk}^x(Z_2)-\mu_v\}
\right]
\right|\\
&\quad\leq
\left\|
\mathbb E_{Z_1}\left[
g_1(Z_1)\{u_{hk}^x(Z_1)-\mu_u\}
\right]
\right\|
\|\hat\Omega_x^{-1}\|
\left\|
\mathbb E_{Z_2}\left[
g_2(Z_2)\{v_{hk}^x(Z_2)-\mu_v\}
\right]
\right\|.
\end{align*}

To bound the first norm, we have
\begin{align*}
&\left\|
\mathbb E_Z\left[
g_1(Z)\{u_{hk}^x(Z)-\mu_u\}
\right]
\right\|^2\\
&\quad\leq
2\left\|
\mathbb E_Z\{g_1(Z)u_{hk}^x(Z)\}
\right\|^2
+
2\|\mu_u\|^2
\left|\mathbb E_Z\{g_1(Z)\}\right|^2\\
&\quad\lesssim
\sum_{\ell=1}^k
\left[
\mathbb E\left\{
g_1(Z)[u_{hk}^x]_\ell(Z)
\right\}
\right]^2
+
\mathbb E\{g_1(Z)^2\}\\
&\quad\leq
\sum_{\ell=1}^k
\mathbb E\left\{[u_{hk}^x]_\ell(Z)^2\right\}
\mathbb E\left[
g_1(Z)^2
\mathbbm 1\left\{
Z\in\operatorname{Supp}([u_{hk}^x]_\ell)
\right\}
\right]
+1\\
&\quad\lesssim
h^{-d}
\mathbb E\left[
g_1(Z)^2
\sum_{\ell=1}^k
\mathbbm 1\left\{
X\in\operatorname{Supp}([b_{hk}^x]_\ell)
\right\}
\right]
+1\\
&\quad\lesssim h^{-d}.
\end{align*}
The second line uses \eqref{eq::mu-u} and Cauchy--Schwarz, the third uses coordinatewise Cauchy--Schwarz, and the fourth uses \eqref{eq::uj} together with
\[
\left\{
Z\in\operatorname{Supp}([u_{hk}^x]_\ell)
\right\}
\subseteq
\left\{
X\in\operatorname{Supp}([b_{hk}^x]_\ell)
\right\}.
\]
The final inequality follows from condition~\eqref{eq::b-support} and $\mathbb E\{g_1(Z)^2\}\leq1$. An analogous argument gives
\[
\left\|
\mathbb E_Z\left[
g_2(Z)\{v_{hk}^x(Z)-\mu_v\}
\right]
\right\|^2
\lesssim h^{-d}.
\]
Since $\|\hat\Omega_x^{-1}\|\lesssim1$, it then follows that
\[
\left|
\mathbb E\left\{
[\tilde\xi^x]_j(Z_1,Z_2)
g_1(Z_1)g_2(Z_2)
\right\}
\right|
\lesssim h^{-d}.
\]
\end{proof}

\subsection{Proof for Proposition~\ref{prop::l_inf_2}}
\label{app::l_inf_2}
\begin{proof}
Denote the kernel of $\hat{S}_x$ by
\begin{align*}
\xi^x(Z_1,Z_2)
={}&\rho_h^x(X_1)K_{hk}^x(X_1)
\left\{
\hat{\psi}_{y_1}(Z_1)
-\hat{\psi}_{a_1}(Z_1)\rho_h^x(X_1)^T\theta_x
\right\}\\
&+\rho_h^x(X_1)K_{hk}^x(X_1)
\left\{
\hat{\psi}_{y_2}^x(Z_1,Z_2)
-\hat{\psi}_{a_2}^x(Z_1,Z_2)\rho_h^x(X_2)^T\theta_x
\right\}K_{hk}^x(X_2).
\end{align*}
In the kernel-weighted terms, it suffices to consider $X\in\mathcal C_h(x)\cup\mathcal C_h(x')$, since both kernels vanish outside this union. For any $x,x'\in\mathcal X$ satisfying $\|x-x'\|\lesssim M^{-1/d}$, the stated condition gives $M^{-1/d}=o(h)$. For $X\in\mathcal C_h(x)\cup\mathcal C_h(x')$, both $(X-x)/h+1/2$ and $(X-x')/h+1/2$ lie in a fixed bounded region, on which $\rho$ is uniformly bounded and Lipschitz. Hence,
\[
\sup_{X\in {\mathcal C_h(x)\cup\mathcal C_h(x')}}
\|\rho_h^x(X)-\rho_h^{x'}(X)\|
\lesssim \frac{M^{-1/d}}h.
\]
Moreover, condition~\eqref{eq::K-lip} implies
\[
\sup_X|K_{hk}^x(X)-K_{hk}^{x'}(X)|
\lesssim
\frac{k^2}{h^{d+1}}\|x-x'\|
\lesssim
\frac{k^2M^{-1/d}}{h^{d+1}}.
\]
Combining the first two bounds and noting that both kernels vanish outside $\mathcal C_h(x)\cup\mathcal C_h(x')$ gives
\begin{align*}
&\sup_X
\|\rho_h^x(X)K_{hk}^x(X)
-\rho_h^{x'}(X)K_{hk}^{x'}(X)\|\\
&\quad\lesssim
h^{-d}
\sup_{X\in\mathcal C_h(x)\cup\mathcal C_h(x')}
\|\rho_h^x(X)-\rho_h^{x'}(X)\|
+\sup_X|K_{hk}^x(X)-K_{hk}^{x'}(X)|\\
&\quad\lesssim
\frac{k^2M^{-1/d}}{h^{d+1}}.
\end{align*}
For a fixed reference basis, both $(X-x)/h+1/2$ and $(X-x')/h+1/2$ belong to $[0,1]^d$ on $\mathcal C_h(x)\cap\mathcal C_h(x')$. Thus condition~\eqref{eq::b-lip} gives
\[
\sup_{X\in {\mathcal C_h(x)\cap\mathcal C_h(x')}}
\|b_{hk}^x(X)-b_{hk}^{x'}(X)\|
\lesssim \frac{k^LM^{-1/d}}h.
\]
On $\mathcal C_h(x)\cap\mathcal C_h(x')$, use the decomposition
\begin{align*}
&b_{hk}^xK_{hk}^x(b_{hk}^x)^T
-b_{hk}^{x'}K_{hk}^{x'}(b_{hk}^{x'})^T=
(b_{hk}^x-b_{hk}^{x'})K_{hk}^x(b_{hk}^x)^T
+b_{hk}^{x'}(K_{hk}^x-K_{hk}^{x'})(b_{hk}^x)^T
+b_{hk}^{x'}K_{hk}^{x'}(b_{hk}^x-b_{hk}^{x'})^T.
\end{align*}
Outside this intersection, at most one kernel is nonzero, so the norm of the matrix difference on the left-hand side above is bounded by $Ck|K_{hk}^x-K_{hk}^{x'}|$. Together with condition~\eqref{eq::b-sup}, this yields
\begin{align*}
\|\hat\Omega_x-\hat\Omega_{x'}\|
&\le
\int
\left\|
b_{hk}^xK_{hk}^x(b_{hk}^x)^T
-b_{hk}^{x'}K_{hk}^{x'}(b_{hk}^{x'})^T
\right\|\,d\hat F(X)\\
&\lesssim
\left(
\frac{k^{L+1/2}M^{-1/d}}{h^{d+1}}
+\frac{k^3M^{-1/d}}{h^{d+1}}
\right)\int d\hat F(X)\\
&\lesssim
\frac{k^{L+3}M^{-1/d}}{h^{d+1}}.
\end{align*}
Applying the same argument to $Q_x$ and $R_x$ gives
\[
\|Q_x-Q_{x'}\|
\lesssim \frac{k^2M^{-1/d}}{h^{d+1}},
\qquad
\|R_x-R_{x'}\|
\lesssim \frac{k^2M^{-1/d}}{h^{d+1}}.
\]
Since $\|\hat\Omega_x^{-1}\|$ and $\|Q_x^{-1}\|$ are uniformly bounded, the identity $A^{-1}-B^{-1}=A^{-1}(B-A)B^{-1}$ implies
\[
\|\hat\Omega_x^{-1}-\hat\Omega_{x'}^{-1}\|
\lesssim
\frac{k^{L+3}M^{-1/d}}{h^{d+1}},
\qquad
\|Q_x^{-1}-Q_{x'}^{-1}\|
\lesssim
\frac{k^2M^{-1/d}}{h^{d+1}}.
\]
{The local mass bound gives $\sup_x\|R_x\|\lesssim1$, and hence $\sup_x\|\theta_x\|\lesssim1$.} Recalling that $\theta_x=Q_x^{-1}R_x$,
\begin{align*}
\|\theta_x-\theta_{x'}\|
&\le
\|Q_x^{-1}-Q_{x'}^{-1}\|\,\|R_x\|
+\|Q_{x'}^{-1}\|\,\|R_x-R_{x'}\| \lesssim
\frac{k^2M^{-1/d}}{h^{d+1}}.
\end{align*}
Furthermore,
\begin{align*}
&\sup_{ {X_1,X_2\in
\mathcal C_h(x)\cap\mathcal C_h(x')}}
\left|
b_{hk}^x(X_1)^{ {T}}
\hat\Omega_x^{-1}b_{hk}^x(X_2)
-
b_{hk}^{x'}(X_1)^{ {T}}
\hat\Omega_{x'}^{-1}b_{hk}^{x'}(X_2)
\right|\\
&\quad\lesssim
\sqrt{k}
\sup_{X\in {\mathcal C_h(x)\cap\mathcal C_h(x')}}
\|b_{hk}^x(X)-b_{hk}^{x'}(X)\|
+k\|\hat\Omega_x^{-1}-\hat\Omega_{x'}^{-1}\|\\
&\quad\lesssim
\frac{k^{L+4}M^{-1/d}}{h^{d+1}}.
\end{align*}
For the second-order terms in $\xi^x(Z_1, Z_2)$, we apply the preceding bilinear bound when $X_1,X_2\in\mathcal C_h(x)\cap\mathcal C_h(x')$. If $X_1$ lies outside this intersection, at least one kernel at $X_1$ vanishes, so the norm of the difference between the second-order terms is bounded by
\[
\frac{Ck}{h^d}
\|\rho_h^x(X_1)K_{hk}^x(X_1)
-\rho_h^{x'}(X_1)K_{hk}^{x'}(X_1)\| .
\]
If $X_1$ lies in the intersection but $X_2$ does not, at least one kernel at $X_2$ vanishes, giving the bound
\[
\frac{Ck}{h^d}
|K_{hk}^x(X_2)-K_{hk}^{x'}(X_2)|.
\]
Combining these cases with the preceding bounds and the uniform boundedness of $Y$, $\hat\pi$, and $\hat\mu_0$ gives
\begin{align*}
&\mathbb E\Biggl[\sup_{\substack{x,x'\in\mathcal X\\ \|x-x'\|\lesssim M^{-1/d}}}
\|\xi^x(Z_1,Z_2)-\xi^{x'}(Z_1,Z_2)\|^2 \;\Bigm|\;D^n\Biggr]\\
&\quad\lesssim
\mathbb E\Biggl[\sup_{\substack{x,x'\in\mathcal X\\ \|x-x'\|\lesssim M^{-1/d}}}
\Biggl\{
\left(1+\frac{k^2}{h^{2d}}\right)
\|\rho_h^x(X_1)K_{hk}^x(X_1)-\rho_h^{x'}(X_1)K_{hk}^{x'}(X_1)\|^2\\
&\qquad+
\frac{\mathbbm{1}\{X_1\in\mathcal C_h(x)\cup\mathcal C_h(x')\}}{h^{2d}}
\bigl|\rho_h^x(X_1)^T\theta_x-\rho_h^{x'}(X_1)^T\theta_{x'}\bigr|^2\\
&\qquad+
\frac{k^2}{h^{2d}}
\Bigl\{
\bigl|\rho_h^x(X_2)^T\theta_xK_{hk}^x(X_2)
-\rho_h^{x'}(X_2)^T\theta_{x'}K_{hk}^{x'}(X_2)\bigr|^2
+|K_{hk}^x(X_2)-K_{hk}^{x'}(X_2)|^2
\Bigr\}\\
&\qquad+
\frac{\mathbbm{1}\{X_1,X_2\in\mathcal C_h(x)\cap\mathcal C_h(x')\}}{h^{4d}}
\Bigl|b_{hk}^x(X_1)^T\hat\Omega_x^{-1}b_{hk}^x(X_2)
-b_{hk}^{x'}(X_1)^T\hat\Omega_{x'}^{-1}b_{hk}^{x'}(X_2)\Bigr|^2
\Biggr\} \;\Bigm|\;D^n\Biggr]\\
&\quad\lesssim
(1+\frac{k^2}{h^{2d}})\left(\frac{k^2M^{-1/d}}{h^{d+1}}\right)^2
+\frac{k^2}{h^{2d}}\left(\frac{k^2M^{-1/d}}{h^{2d+1}}\right)^2
+\frac1{h^{4d}}\left(\frac{k^{L+4}M^{-1/d}}{h^{d+1}}\right)^2\\
&\quad\lesssim
\left(\frac{k^{L+4}M^{-1/d}}{h^{3d+1}}\right)^2.
\end{align*}
Since $\hat S_x=\mathbb U_n(\xi^x)$, Jensen's inequality transfers this bound to $\hat S_x-\hat S_{x'}$. Therefore, when $M^{1/d}\gtrsim nk^{L+7/2}/h^{2d+1}$,
\[
\mathbb E\left[
\sup_{\substack{x,x'\in\mathcal X\\
\|x-x'\|\le\sqrt d\,M^{-1/d}}}
\|\hat S_x-\hat S_{x'}\|^2
\;\middle|\;D^n
\right]
\lesssim \frac{k}{n^2h^{2d}}.
\]
\end{proof}

\section{Discussion of conditions in Theorem~\ref{thm::upper}}

\subsection{Eigenvalue conditions for $Q_x$, $\Omega_x$, and $\hat{\Omega}_x$}
\label{app::eigenvalues}
\begin{proposition} \label{prop:eigen}
Suppose the following conditions hold
    \begin{itemize}
        \item[1.] There exists some constant $0 < \varepsilon < 1/2$ such that $\varepsilon < \pi(x) < 1-\varepsilon$ uniformly for $x \in \mathcal{X}$
        \item[2.] The covariate density $dF(X)$ is bounded above and below away from zero on the support $\mathcal{X}$, and $\int \mathbbm{1}\{X \in \mathcal{C}_h(x)\} dF(X) \asymp h^d$.
        \item[3.] The second-order basis $b_k$ satisfies the required conditions in (\ref{eq::b-L2}-\ref{eq::b-support}), its Gram matrix $\int_{[0,1]^d} b_k(v) b_k(v)^T dv $ has eigenvalues bounded above and below away from zero, and
        \[
        \lambda_{\min} \left(
        \int_{[0,1]^d} b_k(v) b_k(v)^T dF^{*,x}(v)
        \right)
        \gtrsim \lambda_{\min} \left(
        \int_{[0,1]^d} b_k(v) b_k(v)^T dv
        \right).
        \]
    \end{itemize}
Then, for sufficiently large $k$, the eigenvalues of $Q_x$ and $\Omega_x$ are uniformly bounded above and away from zero over $x\in\mathcal X$.
\end{proposition}

In particular, a direct uniform analogue of the sufficient conditions in Propositions S5 and S6 of \cite{kennedy2024minimax} would require the covariate density to be uniformly bounded above and away from zero throughout $\mathcal C_h(x)$. Under the stated basis requirements, this stronger density condition implies Conditions 2 and 3. Our formulation additionally permits holes and disconnected support within local windows, such as the separated-cube supports used in the lower-bound construction. In such settings, the basis may be adapted to the geometry of the covariate support. Further discussion and explicit basis constructions are provided in the next subsection.

\begin{proof}
Let $a\in\mathbb R^q$ be a unit vector. Under the bounded density assumption and orthonormality of $\rho$,
    \[
    \sup_{x\in\mathcal X}a^TQ_xa
    \lesssim \int_{[0,1]^d}(a^T\rho(v))^2\,dv
    =\|a\|^2.
    \]
Let $C_0$ bound the density of $F^{*,x}$ from above and $c_0>0$ bound its total mass $\int_{[0,1]^d}dF^{*,x}(v)$ from below, uniformly in $x$. By orthonormality of $\rho$ and Theorem~8 of \cite{carbery2001distributional}, uniformly over $\|a\|=1$,
    \[
    \int_{[0,1]^d}\mathbbm{1}\{(a^T\rho(v))^2<t\}\,dv
    \leq Ct^{1/(2q)}.
    \]
Choose a sufficiently small fixed $\varepsilon>0$ such that
    \[
    \int_{\mathcal C_{1-\varepsilon}(1/2)}dF^{*,x}(v)
    \geq c_0-C_0\{1-(1-\varepsilon)^d\}\geq c_0/2.
    \]
By Condition~1 and the construction of $K_k$, for sufficiently large $k$,
    \begin{align*}
    a^TQ_xa
    &\gtrsim
    \int_{\mathcal C_{1-\varepsilon}(1/2)}
    (a^T\rho(v))^2\,dF^{*,x}(v)\\
    &\geq t\int_{\mathcal C_{1-\varepsilon}(1/2)}
    \mathbbm{1}\{(a^T\rho(v))^2\geq t\}\,dF^{*,x}(v)\\
    &\geq t\{c_0/2-C_0Ct^{1/(2q)}\}
    \geq tc_0/4>0,
    \end{align*}
where the last inequality follows by choosing $t=(c_0/(4C_0C))^{2q}$. The bound is uniform in $x$ and $\|a\|=1$, so $\inf_{x\in\mathcal X}\lambda_{\min}(Q_x)\gtrsim1$.

For $\Omega_x$, let $\tilde a\in\mathbb R^k$ be a unit vector. Then
    \begin{align*}
    \sup_{x\in\mathcal X}\tilde a^T\Omega_x\tilde a
    &=\sup_{x\in\mathcal X}
    \int_{[0,1]^d}K_k(v)(\tilde a^Tb_k(v))^2\,dF^{*,x}(v)\\
    &\lesssim \int_{[0,1]^d}(\tilde a^Tb_k(v))^2\,dv
    \lesssim \|\tilde a\|^2,
    \end{align*}
since the eigenvalues of $\int_{[0,1]^d}b_k(v)b_k(v)^T\,dv$ are bounded above. For the minimum eigenvalue, by the construction of $K_k$, Condition~3, and the bounded density assumption,
    \begin{align*}
    \tilde a^T\Omega_x\tilde a
    &\geq
    \int_{\mathcal C_{1-1/k^2}(1/2)}
    (\tilde a^Tb_k(v))^2\,dF^{*,x}(v)\\
    &=
    \int_{[0,1]^d}(\tilde a^Tb_k(v))^2\,dF^{*,x}(v)
    -
    \int_{[0,1]^d\setminus\mathcal C_{1-1/k^2}(1/2)}
    (\tilde a^Tb_k(v))^2\,dF^{*,x}(v)\\
    &\gtrsim \|\tilde a\|^2
    - C_0 \sup_{v\in[0,1]^d}\|b_k(v)\|^2\|\tilde a\|^2
    \{1-(1-1/k^2)^d\}\\
    &\gtrsim (1-O(1/k))\|\tilde a\|^2
    \gtrsim \|\tilde a\|^2
    \end{align*}
for sufficiently large $k$. Hence the eigenvalues of $\Omega_x$ are uniformly bounded above and away from zero.
\end{proof}

The eigenvalues of $\hat{\Omega}_x$ are uniformly bounded above and away from zero provided those of $\Omega_x$ are and the density ratios $d\hat F/dF$ and $dF/d\hat F$ are uniformly bounded. This is established in the following Lemma.

\begin{lemma} \label{lemma::eigen-omegahat}
Assume:
    \begin{itemize}
        \item[1.] The eigenvalues of $\Omega_x$ are bounded above and below away from zero uniformly for $x \in \mathcal{X}$
        \item[2.] $\|d\hat{F}/dF\|_{L_\infty(\mathcal{X})}$ and $\|d{F}/d\hat{F}\|_{L_\infty(\mathcal{X})}$ are bounded above.
    \end{itemize}
Then the eigenvalues of $\hat{\Omega}_x$ are bounded above and away from zero uniformly for $x \in \mathcal{X}$
\end{lemma}
\begin{proof}
The density-ratio assumptions imply that, uniformly over $x\in\mathcal X$,
\[
c\,\Omega_x
\preceq
\hat\Omega_x
\preceq
C\,\Omega_x
\]
for some constants $0<c<C<\infty$. The result therefore follows from the assumed uniform eigenvalue bounds for $\Omega_x$. See Lemma~S2 of \cite{kennedy2024minimax} for a detailed argument.
\end{proof}
\subsection{Covariate density and basis constructions}
\label{subsec:density_and_construction}

We next discuss sufficient conditions on the covariate density and support, together with basis constructions satisfying the eigenvalue conditions in Proposition~\ref{prop:eigen} and the basis requirements \eqref{eq::b-L2}--\eqref{eq::b-row-sum}.

\begin{assumption}\label{assump::density}
For the relevant choices of $h$ and $k$ considered in the paper, assume that the following statements hold uniformly over $x\in\mathcal X$:
\begin{itemize}
\item[1.]
The covariate density is bounded above and away from zero on $\mathcal X$, and $\int \mathbbm{1}\{X\in\mathcal C_h(x)\}\,dF(X) \asymp h^d$.

\item[2.]
The support $\mathcal X$ is a union of disjoint, axis-aligned cubes with uniformly comparable volumes. Let $M_x$ denote the number of these cubes intersecting $\mathcal C_h(x)$, and assume $M_x\lesssim k$ uniformly over $x\in\mathcal X$.

\item[3.]
For some fixed constant $c_0>0$, whenever the support $\mathcal{X}$ consists of at least two cubes, the distance between any two distinct support cubes is at least $c_0h/k^{1/d}$.
\end{itemize}
\end{assumption}

In particular, Assumption~\ref{assump::density} allows the covariate support to be either a single cube or a disconnected union of cubes, encompassing both the uniform distribution on $[0,1]^d$ and the covariate density class used in the lower-bound construction. Below, we first consider the simpler full-local-support setting, where $\mathcal C_h(x)$ is contained completely within a support cube, using a tensor-product normalized B-spline basis. We then extend the construction to disconnected local supports using localized spline bases adapted to the support geometry.

\subsubsection{Full local support}
\label{subsec:full-support}
We first consider the case where $\mathcal C_h(x)$ is entirely contained in the covariate support, such that after mapping $\mathcal C_h(x)$ to $[0,1]^d$, a standard tensor-product normalized B-spline basis can be directly used. We explicitly define the tensor-product normalized B-spline basis and verify the required conditions below.
\paragraph{Verification of conditions (\ref{eq::b-L2}-\ref{eq::b-support})}
Fix the spline order $r=\max\{2,\lfloor\alpha\rfloor+\lfloor\beta\rfloor+ \lfloor\gamma\rfloor+1\}$, which ensures continuity and is sufficient for the approximation arguments. For $k$ sufficiently large, partition each coordinate interval $[0,1]$ into $m = \lfloor k^{1/d} \rfloor-r+1$ equal subintervals. For each coordinate $\ell = 1, \cdots, d$, let
\[
\tilde{B}(v_\ell) = (\tilde{B}_1(v_\ell), \cdots, \tilde{B}_{m+r-1}(v_\ell))^T
\]
denote the univariate B-spline basis of order $r$ associated with the equally spaced knots
\[
0, \frac{1}{m}, \frac{2}{m}, \cdots, 1,
\]
constructed via the Cox-de Boor recursion. For a multi-index $ \boldsymbol{\zeta}  = (\zeta_1, \cdots, \zeta_d)$ with $1 \leq \zeta_{\ell} \leq m+r-1$ for each $\ell$, define
\[
\tilde{b}_{\boldsymbol{\zeta}}(v) = \prod_{\ell=1}^d \tilde{B}_{\zeta_\ell} (v_{\ell}), \quad v = (v_1, \cdots, v_d) \in [0,1]^d.
\]
Enumerating all multi-indices yields the tensor-product spline basis $\{\tilde b_j=\tilde b_{\boldsymbol{\zeta}_j}\}_{j=1}^{(m+r-1)^d}$, whose dimension is $(m+r-1)^d=\lfloor k^{1/d}\rfloor^d\asymp k$. Each basis function $\tilde b_j$ is a piecewise polynomial of degree $r-1$ in each coordinate, with support of volume of order $k^{-1}$. Moreover, at any fixed point $v\in[0,1]^d$, at most $r^d$ basis functions are nonzero (See \cite{de1978practical} for more information about splines).

Define the normalized spline basis by
\[
b_k(v) = \sqrt{k}(\tilde{b}_1(v), \cdots , \tilde{b}_k(v))^T.
\]
By the univariate stability bound in Theorem 9.27 of \cite{schumaker2007spline}, $G=\int b_k(v) b_k(v)^T dv$ has eigenvalues bounded above and away from zero. In particular, for each $j= 1, \cdots, k$,
\[
\int_{[0,1]^d} [b_k]_j(v)^2 dv
= G_{j,j} \lesssim 1.
\]
Moreover
\[
\int_{[0,1]^d}|[b_k]_j(v)|\,dv
=\sqrt{k}\int_{[0,1]^d}|\tilde b_j(v)|\,dv\lesssim k^{-1/2},
\]
since each $\tilde{b}_j(v)$ is uniformly bounded above and has support of volume of order $k^{-1}$. Together with bounded overlap, these properties establish conditions \eqref{eq::b-sup} and \eqref{eq::b-support}. Moreover, standard tensor-product spline approximation results imply condition~\eqref{eq::b-L2} at all smoothness levels used in our proofs, which are covered by the chosen spline order; see, for example, Section~3.1 of \cite{belloni2015some}.

Finally, the piecewise derivatives satisfy $\|\nabla_v\tilde b_j\|_{L_\infty([0,1]^d)} \lesssim k^{1/d}$. Continuity and bounded overlap therefore imply
\[
\|b_k(u)-b_k(v)\|
\lesssim k^{1/2+1/d}\|u-v\|,
\qquad u,v\in[0,1]^d,
\]
establishing condition \eqref{eq::b-lip} with $L=1/2+1/d$.

\paragraph{Verification of conditions in Proposition~\ref{prop:eigen}}
In the full-local-support setting, the uniform density bounds above and away from zero imply
\[
\frac{dF^{*,x}}{dv}(v)\asymp1,
\qquad v\in[0,1]^d.
\]
Together with the previously established eigenvalue bounds for the Gram matrix, this verifies the requirements in Condition 3 of Proposition~\ref{prop:eigen}. Thus, the eigenvalues of $\Omega_x$ and $Q_x$ are bounded above and below.

\paragraph{Verification of condition (\ref{eq::b-row-sum})} \label{app::row-sum}
For the constructed tensor-product B-spline basis,
\begin{align*}
[\Omega_x]_{i,j} = & k  \int_{[0,1]^d}  \tilde{b}_{\boldsymbol{\zeta}_i}(v) \tilde{b}_{\boldsymbol{\zeta}_j}(v)  K_k(v) dF^{*,x}(v) \\
= &
 k  \int_{[0,1]^d}  \prod_{\ell=1}^d \tilde{B}_{[\boldsymbol{\zeta}_i]_\ell}(v_\ell) \tilde{B}_{[\boldsymbol{\zeta}_j]_\ell}(v_\ell)  K_k(v) dF^{*,x}(v).
\end{align*}
Since each univariate B-spline has support on at most $r$ consecutive knot intervals,
\[
\tilde{B}_{[\boldsymbol{\zeta}_i]_\ell}(v_\ell)
\tilde{B}_{[\boldsymbol{\zeta}_j]_\ell}(v_\ell)=0
\quad\text{whenever}\quad
|[\boldsymbol{\zeta}_i]_\ell-[\boldsymbol{\zeta}_j]_\ell|\geq r.
\]
Consequently, $[\Omega_x]_{i,j} = 0$ whenever $\| \boldsymbol{\zeta}_i -\boldsymbol{\zeta}_j \|_{\infty} \geq r$. To proceed, we use the following result.
\begin{lemma}[Proposition 5.1 of \cite{demko1984decay}]
\label{lemma:band}
Let $A$ be a positive definite matrix whose eigenvalues lie in $[a,b]$, where $0<a<b$. Define
\[
S_n(A)
=
\bigcup_{\ell=0}^n
\{(i,j):[A^\ell]_{i,j}\ne0\},
\qquad
D_n(A)=(S\times S)\setminus S_n(A),
\]
where $S$ is the index set of $A$. Then
\[
\sup\bigl\{
|[A^{-1}]_{i,j}|:(i,j)\in D_n(A)
\bigr\}
\le C_0q^{n+1}.
\]
where $\kappa= b/a$, $q={(\sqrt{\kappa}-1)}/{(\sqrt{\kappa}+1)}$ and $C_0={(1+\sqrt{\kappa})^2}/{2a\kappa}$.
\end{lemma}

For $i\ne j$, if $[(\Omega_x)^n]_{i,j}=0$ for every $n\ge1$, then Lemma~\ref{lemma:band} gives $[\Omega_x^{-1}]_{i,j}=0$ with $n\to\infty$. Otherwise, define
\[
n_0(i,j)
=
\min\{n\ge1:[(\Omega_x)^n]_{i,j}\ne0\}
=
\max\{n\ge0:(i,j)\in D_n(\Omega_x)\}+1.
\]
By definition, there exist indices
\[
i=i_0,i_1,\ldots,i_{n_0(i,j)}=j
\]
such that $[\Omega_x]_{i_{t-1},i_t}\ne0$ for $t=1,\ldots,n_0(i,j)$. Since $[\Omega_x]_{i,j}=0$ whenever $\|\boldsymbol{\zeta}_i-\boldsymbol{\zeta}_j\|_\infty\ge r$, the triangle inequality gives
\[
\|\boldsymbol{\zeta}_i-\boldsymbol{\zeta}_j\|_\infty
\le
\sum_{t=1}^{n_0(i,j)}
\|\boldsymbol{\zeta}_{i_t}
  -\boldsymbol{\zeta}_{i_{t-1}}\|_\infty
<
n_0(i,j)r.
\]
Hence
\[
n_0(i,j)>
\frac{\|\boldsymbol{\zeta}_i-\boldsymbol{\zeta}_j\|_\infty}{r}.
\]
Applying Lemma~\ref{lemma:band} with $n=n_0(i,j)-1$ yields
\[
|[\Omega_x^{-1}]_{i,j}|
\le C_0q^{n_0(i,j)}
\le
C_0\tilde q^{\|\boldsymbol{\zeta}_i-\boldsymbol{\zeta}_j\|_\infty},
\qquad
\tilde q=q^{1/r}\in(0,1),
\]
and the bound holds for all $i\ne j$. Moreover, $|[\Omega_x^{-1}]_{i,i}|\le a^{-1}$, where $a>0$ is the uniform lower eigenvalue bound. Therefore,
\begin{align*}
\|\Omega_x^{-1}\|_\infty
&=
\max_{1\le i\le k}
\sum_{j=1}^k|[\Omega_x^{-1}]_{i,j}|\\
&\le
a^{-1}
+
C_0\max_{1\le i\le k}
\sum_{j\ne i}
\tilde q^{\|\boldsymbol{\zeta}_i-\boldsymbol{\zeta}_j\|_\infty}\\
&=
a^{-1}
+
C_0\max_{1\le i\le k}
\sum_{t=1}^{\infty}
\#\left\{
j:\|\boldsymbol{\zeta}_i-\boldsymbol{\zeta}_j\|_\infty=t
\right\}\tilde q^t\\
&\lesssim
a^{-1}+C_0\sum_{t=1}^{\infty}t^{d-1}\tilde q^t
\lesssim1,
\end{align*}
where, for $t\ge1$,
\[
\#\left\{
j:\|\boldsymbol{\zeta}_i-\boldsymbol{\zeta}_j\|_\infty=t
\right\}
\le
\#\{\boldsymbol{\zeta}\in\mathbb Z^d:
       \|\boldsymbol{\zeta}\|_\infty=t\}
=
(2t+1)^d-(2t-1)^d
\lesssim t^{d-1}.
\]
The series converges because $\tilde q^t$ decays exponentially. The same argument applies to $\hat\Omega_x$ under its corresponding uniform eigenvalue bounds and yields
\[
\|\hat\Omega_x^{-1}\|_\infty\lesssim1
\]
since it has the same sparsity property. We have therefore verified all the required properties in the full-local-support setting.

\subsubsection{Disconnected local support}

We first consider support cubes with side length $O(h)$, including $o(h)$. Fix an enumeration of the support cubes,
\[
\mathcal X=\bigcup_j A_j
=\bigcup_j\mathcal C_{\ell_j}(m_j),
\]
where $m_j$ and $\ell_j$ denote their centers and side lengths. For $x\in\mathcal X$, let $i(x)$ be the index of the cube containing $x$, and define
\[
\mathcal J_x
=
\{j:m_j\in\mathcal C_h(m_{i(x)})\}
=
\{j_1,\ldots,j_{M_x}\},
\qquad j_1<\cdots<j_{M_x}.
\]
Here $M_x=|\mathcal J_x|\lesssim k$ may depend on $x$. The selected collection remains unchanged as $x$ varies within the same support cube, including its boundary, thereby avoiding truncation of the selected cubes near boundary points. The local mass and comparable-volume conditions imply
\[
\sum_{\nu=1}^{M_x}|A_{j_\nu}|\asymp h^d,
\qquad
\ell_{j_\nu}\asymp\frac{h}{M_x^{1/d}},
\quad \nu=1,\ldots,M_x.
\]
Below, we construct both the kernel and the second-order basis adapted to such disconnected support.
\paragraph{Construction of the spline blocks and kernel}
For each fixed $x$, use the same fixed spline order $r=\max\{2,\lfloor\alpha\rfloor+\lfloor\beta\rfloor +\lfloor\gamma\rfloor+1\}$, and set
\[
m=\left\lceil(k/M_x)^{1/d}\right\rceil,
\qquad
J=(m+r-1)^d\asymp k/M_x.
\]
Let $\widetilde B_J$ denote the unnormalized tensor-product B-spline basis on $[0,1]^d$, constructed by the Cox--de Boor recursion with $m$ equal knot intervals in each coordinate. Each component in $\tilde{B}_J$ takes values in $[0,1]$, and at most $r^d$ components are nonzero at any point. Define the normalized basis by
\[
B_J(u)=\sqrt J\,\widetilde B_J(u).
\]
For each selected cube $A_j$, $j\in\mathcal J_x$, define the corresponding $J$-dimensional block on the covariate support by
\[
b_j(X)=
\mathbbm{1}\{X\in A_j\}
\left(\frac{h}{\ell_j}\right)^{d/2}
B_J\!\left(\frac{X-m_j}{\ell_j}+\frac12\right).
\]
Concatenating these blocks in the fixed order $j_1<\cdots<j_{M_x}$, define
\[
b_{hk}^x(X)
=
\bigl(
b_{j_1}(X)^T,\ldots,b_{j_{M_x}}(X)^T
\bigr)^T,
\]
where coordinates $(\nu-1)J+1,\ldots,\nu J$ correspond to cube $A_{j_\nu}$, and the total dimension is $M_xJ\asymp k$. Directly define the localized kernel by
\[
K_{hk}^x(X)
=
h^{-d}\mathbbm{1}
\left\{X\in\bigcup_{\nu=1}^{M_x}A_{j_\nu}\right\}.
\]
These definitions specify the adapted localized basis and kernel on the covariate support $\mathcal X$. In addition, the projection onto the localized basis with respect to the weighted measure $K_{hk}^x\,dF$ is defined by
\[
\Pi_xg(X)=b_{hk}^x(X)^T\Omega_x^{-1}
\int_{\mathcal X}b_{hk}^x(u)g(u)K_{hk}^x(u)\,dF(u),
\]
with $\hat\Pi_x$ obtained by replacing $\Omega_x^{-1}$ with $\hat\Omega_x^{-1}$.

\paragraph{Verification of the basis and Gram matrix conditions.}
We directly verify the localized counterparts of the basis conditions. Under an $h$-scale change of variables, the pointwise vector norm bounds remain unchanged, whereas the integrals of each coordinate's absolute value and square acquire a factor $h^d$. Accordingly, we normalize these Lebesgue integrals by $h^{-d}$ below. All bounds hold uniformly in $x$.

The normalization factor satisfies $(h/\ell_j)^{d/2}\asymp\sqrt{M_x}$. Since $\sup_u\|B_J(u)\|,\sup_u\|B_J(u)\|_1\lesssim\sqrt J$ and at most one block $b_j(X)$ is nonzero at any given $X\in\mathcal X$,
\[
\sup_{X\in\mathcal X}\|b_{hk}^x(X)\|, \
\sup_{X\in\mathcal X}\|b_{hk}^x(X)\|_1
\lesssim\sqrt{M_xJ}\asymp\sqrt k,
\]
verifying the localized counterpart of condition~\eqref{eq::b-sup}.

Each coordinate after normalization has support volume $O(\ell_{j_\nu}^d/J)=O(h^d/k)$, and at most $r^d$ coordinates are nonzero at any point. Consequently, for every coordinate $j$,
\[
h^{-d}\int_{\mathcal X}[b_{hk}^x]_j(X)^2\,dX
\lesssim1,
\qquad
h^{-d}\int_{\mathcal X}|[b_{hk}^x]_j(X)|\,dX
\lesssim k^{-1/2}.
\]
These establish the localized integral and bounded-overlap conditions corresponding to~\eqref{eq::b-support}.

Within each selected cube $A_j$, write $u=(X-m_j)/\ell_j+1/2$. The fixed-order spline derivative bounds and bounded overlap give $\|\nabla_u B_J(u)\|_{\mathrm{op}}\lesssim\sqrt J\,m$. Thus, wherever the derivatives exist, the chain rule yields
\[
\|\nabla_X b_j(X)\|_{\mathrm{op}}
\lesssim
\left(\frac h{\ell_j}\right)^{d/2}
\frac{\sqrt J\,m}{\ell_j}
\asymp
\sqrt{M_xJ}\,
\frac{(k/M_x)^{1/d}}{h/M_x^{1/d}}
\asymp
\frac{k^{1/2+1/d}}h.
\]
Since the spline is continuous, integrating this bound along the line segment joining $X,Y\in A_j$ gives
\[
\|b_j(X)-b_j(Y)\|
\lesssim
\frac{k^{1/2+1/d}}h\|X-Y\|,
\]
verifying the counterpart of condition~\eqref{eq::b-lip}.

By construction, the normalized Lebesgue Gram matrix
\[
h^{-d}\int_{\mathcal X}
b_{hk}^x(X)b_{hk}^x(X)^T\,dX
\]
is block diagonal. For each selected cube, a change of variables gives
\[
h^{-d}\int_{A_{j_\nu}}
b_{j_\nu}(X)b_{j_\nu}(X)^T\,dX
=
\int_{[0,1]^d}B_J(u)B_J(u)^T\,du.
\]
The spline Gram matrix bounds therefore imply that the eigenvalues of the full normalized Lebesgue Gram matrix are uniformly bounded above and away from zero. For the weighted Gram matrix, distinct blocks have disjoint supports, so
\[
\Omega_x
=
\operatorname{diag}
(\Omega_{x,1},\ldots,\Omega_{x,M_x}),
\]
where
\[
\Omega_{x,\nu}
=
h^{-d}\int_{A_{j_\nu}}
b_{j_\nu}(X)b_{j_\nu}(X)^T dF(X)
=
\int_{[0,1]^d}
B_J(u)B_J(u)^T
f\!\left(m_{j_\nu}
+\ell_{j_\nu}(u-1/2)\right)\,du.
\]
The density and spline Gram matrix bounds imply
\[
\lambda_{\min}(\Omega_{x,\nu})\gtrsim1,
\qquad
\lambda_{\max}(\Omega_{x,\nu})\lesssim1,
\]
uniformly in $x$, $\nu$, and $J$. The same bounds hold for $\Omega_x$, also verifying the localized Gram matrix comparison in Proposition~\ref{prop:eigen}. By the same argument as in Proposition~\ref{prop:eigen}, the eigenvalues of $Q_x$ are uniformly bounded above and away from zero over $x\in\mathcal X$.

Within each block, the spline overlap structure is the same as in the preceding subsection. Applying the argument of Appendix~\ref{app::row-sum} gives
\[
\|\Omega_{x,\nu}^{-1}\|_\infty\lesssim1.
\]
Consequently,
\[
\|\Omega_x^{-1}\|_\infty
=
\max_{1\le\nu\le M_x}
\|\Omega_{x,\nu}^{-1}\|_\infty
\lesssim1.
\]
The same argument applies to $\hat\Omega_x$ under its corresponding uniform eigenvalue bounds, establishing condition~\eqref{eq::b-row-sum}.

We directly verify the localized continuity used in Proposition~\ref{prop::l_inf_2}, for which condition~\eqref{eq::b-lip} was used. If $x,x'$ belong to the same support cube, they have the same selected indices, $M_x=M_{x'}$, the same block dimension $J$, and the same ordering of coordinates. Consequently,
\[
b_{hk}^x(X)=b_{hk}^{x'}(X),
\qquad
K_{hk}^x(X)=K_{hk}^{x'}(X),
\qquad
\hat\Omega_x=\hat\Omega_{x'}.
\]
The corresponding kernel-weighted projection kernels are therefore identical.

If $x,x'$ belong to distinct support cubes, separation condition gives $\|x-x'\|\gtrsim h/k^{1/d}$. Using the sup-norm bounds and $\sup_x\|\hat\Omega_x^{-1}\|\lesssim1$, we obtain
\[
\sup_{X_1,X_2\in\mathcal X}
\left|
K_{hk}^x(X_1)b_{hk}^x(X_1)^T
\hat\Omega_x^{-1}b_{hk}^x(X_2)K_{hk}^x(X_2)
\right|
\lesssim \frac{k}{h^{2d}}.
\]
Combining the two cases therefore yields
\begin{align}
\label{eq::b-local-lip}
&\sup_{X_1,X_2\in\mathcal X}\Bigl|
K_{hk}^x(X_1)b_{hk}^x(X_1)^T
\hat\Omega_x^{-1}b_{hk}^x(X_2)K_{hk}^x(X_2)
\nonumber\\
&\qquad-
K_{hk}^{x'}(X_1)b_{hk}^{x'}(X_1)^T
\hat\Omega_{x'}^{-1}b_{hk}^{x'}(X_2)K_{hk}^{x'}(X_2)
\Bigr|
\lesssim
\frac{k^{1+1/d}}{h^{2d+1}}\|x-x'\|,
\end{align}
and this scalar comparison remains valid when the local basis dimensions differ. By the same argument,
\[
\sup_{X\in\mathcal X}
|K_{hk}^x(X)-K_{hk}^{x'}(X)|
\lesssim
\frac{k^{1/d}}{h^{d+1}}\|x-x'\|
\lesssim
\frac{k^2}{h^{d+1}}\|x-x'\|.
\]
These bounds provide the localized continuity required in the proof of Proposition~\ref{prop::l_inf_2}. The discretization size $M$ can be chosen accordingly, while remaining polynomial in $n$.

For the approximation condition, the spline mesh on each selected cube is $\ell_{j_\nu}/m\asymp h/k^{1/d}$. By the uniform spline approximation bound discussed in Section~3.1 of \cite{belloni2015some}, for every $t$-H\"older function $g$ with uniformly bounded H\"older norm, at a smoothness level covered by the chosen spline order, there exists a spline $s_{j_\nu}$ satisfying
\[
\sup_{X\in A_{j_\nu}}|g(X)-s_{j_\nu}(X)|
\lesssim (\ell_{j_\nu}/m)^t
\asymp (h/k^{1/d})^t,
\]
uniformly in $x$ and $\nu$. Since the blocks have independent coefficients, these approximants can be represented simultaneously by a single linear combination of $b_{hk}^x$. Together with $\int K_{hk}^x\,dF\lesssim1$, this gives
\[
\inf_{c\in\mathbb R^{M_xJ}}
\left\{
\int_{\mathcal X}
\bigl(g(X)-c^Tb_{hk}^x(X)\bigr)^2
K_{hk}^x(X)\,dF(X)
\right\}^{1/2}
\lesssim (h/k^{1/d})^t,
\]
and the weighted orthogonal projection minimizes this error. This establishes the localized counterpart of condition~\eqref{eq::b-L2}.

Finally, consider a support cube $A_j=\mathcal C_{\ell_j}(m_j)$ with $\ell_j/h\to\infty$. The full-local-support construction applies directly whenever $\mathcal C_h(x)\subseteq A_j$. Near the boundary, let $\tilde x$ be the coordinatewise projection of $x$ onto the concentric cube $\mathcal C_{\ell_j-h}(m_j)$. Then
\[
\mathcal C_h(\tilde x)\subseteq A_j,
\qquad
\|\tilde x-x\|\lesssim h,
\qquad
\|\tilde x-\tilde x'\|\leq\|x-x'\|,
\quad x,x'\in A_j.
\]
Apply the full-local-support construction on $\mathcal C_h(\tilde x)$, centering both the second-order spline basis and the kernel at $\tilde x$, while keeping $\rho_h^x$ centered at $x$ and evaluating the fitted polynomial at $x$ as in \eqref{eq::estimator}. Since $\|\tilde{x}-x\| \lesssim h$, the required approximation, eigenvalue, and localized continuity bounds follow by arguments similar to those above.

\subsection{Stability of $\hat Q_x^{-1}$}
\label{app::Q-stability}

Let $\epsilon>0$ be a fixed numerical threshold, and assume
\[
\inf_{x\in\mathcal X}\lambda_{\min}(Q_x)\geq2\epsilon.
\]
Throughout, we condition on the training sample $D^n$. Singular-value perturbation gives
\[
\left\{\inf_{x \in \mathcal{X}} \sigma_{\min}(\hat Q_x)<\epsilon\right\}
\subseteq
\left\{\sup_{x \in \mathcal{X}}\|\hat Q_x-Q_x\|>\epsilon\right\},
\]
where we use the smallest singular value of $\hat Q_x$ because it need not be symmetric. We therefore bound the probability of the latter event.

Following the conditional-mean calculation for $\hat Q$ in the proof of Proposition~5 of \cite{kennedy2024minimax}, {under the remaining assumptions of Theorem~\ref{thm::upper},} we obtain
\begin{align*}
|[\mathbb E(\hat Q_x\mid D^n)-Q_x]_{j,\ell}|
&\lesssim
\|(I-\Pi_x)\{[\rho]_j
(\pi^{*,x}-\hat\pi^{*,x})\}\|_{K\cdot F^{*,x}}
\|(I-\Pi_x)([\rho]_\ell\pi^{*,x})\|_{K\cdot F^{*,x}}\\
&\quad+
\|\hat\pi-\pi\|_{K\cdot F^{*,x}}
\|\hat\Omega_x^{-1}-\Omega_x^{-1}\|\\
&\lesssim
{(h/k^{1/d})^\alpha
+\frac{(h/k^{1/d})^{2s}}{h^\gamma}
+\frac{r_n}{h^\gamma}.}
\end{align*}
Under the stated choices of $h,k$, the ratio $(h/k^{1/d})^{2s}/h^\gamma$ can be made sufficiently small by choosing a sufficiently large fixed multiplier in $k$. Since $(h/k^{1/d})^\alpha\to0$ and $r_n/h^\gamma\to0$, we obtain
\begin{equation}\label{eq::Q-stability-bias}
\sup_x\|\mathbb E(\hat Q_x\mid D^n)-Q_x\|
\leq\epsilon/2
\end{equation}
for all sufficiently large $n$.

For the fluctuation, we follow the concentration and continuity arguments in the proofs of Propositions~\ref{prop::l_inf_1} and~\ref{prop::l_inf_2}, respectively. Write $\hat Q_x=\mathbb U_n(\xi_Q^x)$, where
\begin{align*}
\xi_Q^x(Z_1,Z_2)
={}&\rho_h^x(X_1)K_{hk}^x(X_1)
A_1\{A_1-\hat\pi(X_1)\}\rho_h^x(X_1)^T -\rho_h^x(X_1)K_{hk}^x(X_1)
\{A_1-\hat\pi(X_1)\}\\
&\qquad\times
\{b_{hk}^x(X_1)^T\hat\Omega_x^{-1}b_{hk}^x(X_2)\}
A_2K_{hk}^x(X_2)\rho_h^x(X_2)^T.
\end{align*}
For each entry, the outcome residual in the preceding kernel calculations is replaced by $A[\rho_h^x(X)]_\ell$, which is uniformly bounded on the kernel support. Thus the same bounds in Proposition~\ref{prop::kernel-bound} apply. After symmetrizing in the two observations and applying the Hoeffding decomposition entrywise, we obtain, in the notation of Lemma~\ref{lemma::concentration},
\[
b\lesssim h^{-d},\quad
\sigma^2\lesssim h^{-d},\quad
A\lesssim kh^{-2d},\quad
B^2\lesssim nkh^{-3d},\quad
C^2\lesssim n^2kh^{-2d},\quad
D\lesssim nh^{-d}.
\]
For every fixed $t>0$, Lemma~\ref{lemma::concentration} then gives,
\begin{align}
\nonumber
\Pr\!\left\{
\|\hat Q_x-\mathbb E(\hat Q_x\mid D^n)\|>t
\;\middle|\;D^n
\right\}
& \lesssim
\exp\!\left[
-\frac{1}{C_t}
\min\left\{
\frac{n^2h^{2d}}{k},\,
nh^d,\,
\frac{nh^d}{k^{1/3}},\,
\frac{nh^d}{\sqrt{k}}
\right\}
\right]\\
\label{eq:concen-Q}
&  \lesssim
\exp\!\left(-\frac{nh^d}{C_t\sqrt{k}}\right),
\end{align}
where the last inequality holds for all sufficiently large $n$, since $k=o((nh^d)^2)$ under the stated tuning choices. Choose packing points $x_1,\ldots,x_M$ as in Appendix~\ref{app::decomposition}, so that every $x\in\mathcal X$ lies within distance $O(M^{-1/d})$ of some packing point. Take $M$ polynomial in $n$ satisfying the mesh condition in Proposition~\ref{prop::l_inf_2}. The pointwise kernel bounds in its proof apply to $\xi_Q^x$ and hence to both its empirical average and conditional expectation. Thus,
\begin{align*}
&\sup_{\substack{x,x'\in\mathcal X\\
\|x-x'\|\lesssim M^{-1/d}}}
\left\{
\|\hat Q_x-\hat Q_{x'}\|
+\|\mathbb E(\hat Q_x-\hat Q_{x'}\mid D^n)\|
\right\} \lesssim
\frac{k^{L+4}M^{-1/d}}{h^{3d+1}}
\lesssim\frac{\sqrt{k}}{nh^d}
\leq\frac{\epsilon}{4},
\end{align*}
for sufficiently large $n$. Together with \eqref{eq::Q-stability-bias}, this gives
\begin{align*}
\sup_x\|\hat Q_x-Q_x\|
&\leq
\sup_x\|\hat Q_x-\mathbb E(\hat Q_x\mid D^n)\|
+\frac{\epsilon}{2}\\
&\leq
\max_{1\leq\ell\leq M}
\|\hat Q_{x_\ell}
-\mathbb E(\hat Q_{x_\ell}\mid D^n)\|
+\frac{3\epsilon}{4}.
\end{align*}
Applying the pointwise tail bound in \eqref{eq:concen-Q} with $t=\epsilon/4$ and taking a union bound yields
\begin{align*}
&\Pr\!\left\{
\sup_x\|\hat Q_x-Q_x\|>\epsilon
\;\middle|\;D^n
\right\} \leq
\Pr\!\left\{
\max_{1\leq\ell\leq M}
\|\hat Q_{x_\ell}
-\mathbb E(\hat Q_{x_\ell}\mid D^n)\|>\epsilon/4
\;\middle|\;D^n
\right\} \lesssim
M\exp\!\left(-\frac{nh^d}{C_\epsilon\sqrt{k}}\right).
\end{align*}
The conditional bound is uniform in $D^n$, so integrating over the training sample gives
\begin{align*}
\Pr\!\left\{\inf_x\sigma_{\min}(\hat Q_x)<\epsilon\right\}
&\leq
\Pr\!\left\{\sup_x\|\hat Q_x-Q_x\|>\epsilon\right\} \lesssim
M\exp\!\left(-\frac{nh^d}{C_\epsilon\sqrt{k}}\right).
\end{align*}
Under the stated tuning choices, $nh^d/\sqrt{k}$ grows at least as a positive power of $n$, while $M$ is polynomial in $n$. Thus the last bound is $O(\exp(-c n^\eta))$ for some constants $c,\eta>0$, and hence decays faster than every polynomial rate.

\section{Preliminary results}
\subsection{Maximum inequality from tail bounds} \label{pre::tail}
\begin{lemma}  \label{lemma::maximum}
Let $M\ge2$ and let $X_1, \cdots, X_M \in \mathbb{R}$ be random variables satisfying, for every $i=1,\ldots,M$ and every $t>0$,
    \[
    \Pr(|X_i| \geq t) \lesssim e^{-c\min
    \left(
    \frac{t^2}{C^2}, \frac{t}{D}, \frac{t^{2/3}}{B^{2/3}}, \frac{t^{1/2}}{A^{1/2}}
    \right) ,}
    \]
for some constant $c>0$. Then,
    \[
    \mathbb{E}[\max_{1\leq i \leq M} |X_i|^2 ] \lesssim
    C^2 {\log(M)} + D^2 (\log(M))^2 + B^2(\log M)^{3} +A^2(\log M)^4.
    \]
\end{lemma}
\begin{proof}
By rescaling the constants if necessary, we may assume
    \[
     \Pr(|X_i| \geq t) \lesssim e^{-2\min
    \left(
    \frac{t^2}{\tilde{C}^2}, \frac{t}{\tilde{D}}, \frac{t^{2/3}}{\tilde{B}^{2/3}}, \frac{t^{1/2}}{\tilde{A}^{1/2}}
    \right)}
    \]
where $\tilde{A} \asymp A$, $\tilde{B} \asymp B$, $\tilde{C} \asymp C$, and $\tilde{D} \asymp D$. Using the identity $\mathbb{E}(Y) = \int_0^\infty \Pr(Y\ge t)\,dt $ for nonnegative random variables and applying the union bound, we obtain
    \begin{align*}
        \mathbb{E}(\max_{1 \leq i \leq M} |X_i|^2)
        \leq &
        \int_0^\infty \min\left(1, \sum_{i=1}^M \Pr ( |X_i| \geq \sqrt{t})\right) dt
        \leq
        \int_0^\infty 2t\min\left\{1,M\max_i\Pr(|X_i|\ge t)\right\}\,dt.
    \end{align*}
Substituting the tail bound yields
    \[
        \mathbb{E}(\max_{1 \leq i \leq M} |X_i|^2)
        \lesssim I_C + I_D + I_B + I_A,
    \]
where
    \[
        I_C =  \int_0^\infty t \min(1, M e^{-{\frac{2t^2}{\tilde{C}^2}}})dt, \quad
        I_D = \int_0^\infty t \min(1, M e^{-{\frac{2t}{\tilde{D}}}}) dt,
    \]
    \[
        I_B = \int_0^\infty t \min(1, M  e^{-{\frac{2t^{2/3}}{\tilde{B}^{2/3}}}}) dt, \quad
        I_A = \int_0^\infty t \min(1, M
        e^{-{\frac{2t^{1/2}}{\tilde{A}^{1/2}}}}) dt.
    \]
We now bound each integral separately. Splitting the four integrals at $\tilde{C} \sqrt{\log(M)/2}$, $\tilde{D} \log(M)/2$, $\tilde{B} \log(M)^{3/2}/2^{3/2}$, and $\tilde{A} \log(M)^2/2^2$ respectively, and applying the corresponding change of variables gives
    \begin{align*}
        I_C & \lesssim
         \int_{0}^{\tilde{C} \sqrt{\log M/2}} tdt +
        M {\tilde{C}^2} \int_{\sqrt{\log M}}^\infty  ue^{-u^2} du \ \lesssim \tilde{C}^2 {\log(M)}, \\
        I_D&  \lesssim
        \int_{0}^{\tilde{D}\log M/2} t dt + M \tilde{D}^2  \int_{\log M}^\infty ue^{-u} du \lesssim \tilde{D}^2 (\log(M))^2, \\
        I_B &  \lesssim
        \int_{0}^{ \tilde{B}(\log M)^{3/2} / 2^{3/2} } t dt + M \tilde{B}^2 \int_{\log M}^{\infty}u^{2} e^{-u}  du
         \lesssim \tilde{B}^2 (\log M)^{3}, \\
        I_A  &   \lesssim
        \int_{0}^{\tilde{A}(\log M)^2/2^2} t dt +  M \tilde{A}^2 \int_{\log M}^{\infty}
          u^3 e^{-u} du \lesssim \tilde{A}^2 (\log M)^4.
    \end{align*}
Combining the four bounds yields the stated result.
\end{proof}
Further, the same argument applies when the tail bound contains only a subset of the four components, in which case the corresponding terms may be omitted from the final bound.

\section*{AI statement}
The authors used OpenAI's ChatGPT to assist with language editing and readability. The authors reviewed and verified all content and take full responsibility for the final manuscript.

\bibliography{bib}

\end{document}